\documentclass[reqno,12p]{amsart}
\usepackage[T1]{fontenc}

\usepackage{geometry}
\usepackage{mathtools,amssymb,amsthm,mathrsfs,color,lineno,paralist,graphicx,float}

\def\black{\color{black}}
\def\tP{\widetilde P}
\definecolor{orange}{rgb}{1,0.5,0}

\definecolor{green}{rgb}{0.5,0.7,0.3}

\usepackage[colorlinks,
linkcolor=red,
anchorcolor=green,
citecolor=blue,
]{hyperref}
\usepackage{amsmath}
\usepackage{pdfpages}
\usepackage{amstext}
\usepackage{stmaryrd}
\usepackage{bbm}
\usepackage{comment}
\usepackage{faktor,url}
\usepackage{xfrac}
\usepackage{mathabx}
\usepackage{enumitem}

\usepackage[normalem]{ulem}

\numberwithin{equation}{section}
\numberwithin{figure}{section}
\theoremstyle{plain}
\newtheorem{thm}
{\protect\theoremname}[section]
  \theoremstyle{definition}
  \newtheorem{defn}[thm]{\protect\definitionname}
  \theoremstyle{remark}
  \newtheorem{rem}{\protect\remarkname}
  \theoremstyle{plain}
  \newtheorem{lem}[thm]{\protect\lemmaname}
  \theoremstyle{plain}
  \newtheorem{prop}[thm]{\protect\propositionname}
  \theoremstyle{plain}

\newtheorem{thmx}{Theorem}

\newtheorem{thmy}{Theorem}

  \providecommand{\definitionname}{Definition}
  \providecommand{\lemmaname}{Lemma}
  \providecommand{\propositionname}{Proposition}
  \providecommand{\remarkname}{Remark}
\providecommand{\theoremname}{Theorem}
\providecommand{\corollaryname}{Corollary}

\newcommand{\bs}{\boldsymbol}

\title[KAM for  weakly interacting particles I: full-dimensional tori]
{Kolmogorov invariant torus theorem for  weakly interacting particles I: Full dimensional tori}

\author[D. Dolgopyat]{D. Dolgopyat}
\address[DD]{Department of Mathematics, University of Maryland, 4176 Campus Drive, 20782, MD, USA}
\email{dolgop@umd.edu}

\author[B.Fayad]{B. Fayad}
\address[BF]{Department of Mathematics, University of Maryland, 4176 Campus Drive, 20782, MD, USA}
\email{bassam@umd.edu}

\author[J. Paradela]{J. Paradela}
\address[JP]{ Department of Mathematics, University of Maryland, 4176 Campus Drive, 20782, MD, USA}
\email{paradela@umd.edu}

\begin{document}

\begin{abstract} 
We develop an abstract KAM theorem for systems of infinitely many interacting particles with decaying masses and all-to-all interactions. Using this framework, we construct full-dimensional KAM tori for infinite-dimensional mechanical systems exhibiting long range interactions. 
\end{abstract}

\maketitle
\tableofcontents

\section{Introduction}\label{sec:intro}

The stability of motion in finite-dimensional systems of particles interacting {\it via} a small potential has been extensively studied since the advent of Kolmogorov-Arnold-Moser (KAM) theory (see, for example, \cite{AKN}). 
For infinite-dimensional Hamiltonian systems, a work of remarkable interest is \cite{PoschelSpatialstructure} (see also \cite{Frohlichlocalization}), in which Pöschel develops an abstract functional framework to construct full-dimensional invariant tori for systems of the form
\begin{equation}\label{eq:latticesystem}
H(\theta,I)=\langle \omega,I\rangle+\varepsilon P(\theta,I),
\qquad (\theta,I)\in (\mathbb{T}\times\mathbb{R})^\Lambda,
\end{equation}
where $\Lambda$ is a lattice (of arbitrary dimension), $\langle \omega,I\rangle=\sum_{i\in\Lambda} \omega_i I_i$, and the perturbation $P$ possesses a property called ``spatial structure''. This property concerns the \textit{spatial distribution} in $\Lambda$ of the support of the monomials in $P$, together with a \textit{weight function} that assigns a weight to each monomial depending on the size and location of its support. The crucial feature of the spatial structure property is that it is preserved under the Poisson bracket operator and that functions with spatial structure generate vector fields that decay sufficiently fast.

In particular, P\"oschel's theory does not encompass long range interactions (in the sense of Definition \ref{defn:longrange} below). The main goal of this work is to develop an abstract KAM theorem in infinite dimensions in the context of long range interactions. 
 The primary application of our theory is the construction of invariant tori for mechanical systems consisting of infinitely many interacting particles with decaying masses and all-to-all interactions. As we will show below, such systems cannot be analyzed using the techniques of \cite{PoschelSpatialstructure} and therefore require the incorporation of new ideas.

The present work forms part of a series of two papers. Here we deal with the existence of full-dimensional tori for the class of systems described below. In the second part \cite{part2}, under suitable assumptions, we extend our results to normally-elliptic tori and show that these are accumulated by full-dimensional ones.

\subsection{Main results: informal statements}\label{sec:informal}
In this section we present a non-technical statement of our main result. It is  formulated for an abstract normal form, allowing for an easier comparison with \eqref{eq:latticesystem} and providing the reader with a clearer perspective on how our results relate to those previously obtained in the literature (see Section \ref{sec:comparison}, where we also include a comparison with a specific PDE setting). A more precise formulations of this abstract theorem is given in Section \ref{sec:mainabstract}.

\subsection*{Long range perturbations}  We denote by  $\ell_\infty$ the space of complex valued sequences with bounded supremum norm and write $\ell_\infty^\mathbb R$ when considering real sequences. Fix any value of $\kappa>0$ and let \begin{equation}\label{eq:spacedecayingmasses}
\ell^\mathbb R_{\infty,\kappa}=\{\boldsymbol m=\{m_i\}_{i\in\mathbb N}\subset (0,\infty) \colon \sup_{i\in\mathbb N} m_i\ \exp(\kappa i)<\infty\}.
\end{equation}
We may, and will assume, that $m_i$ are non-increasing. The Hamiltonians introduced below are defined on an infinite-dimensional phase space $\mathcal T^\infty\times\ell_{\infty,\bs m}$ defined as follows \black. On one hand, denote
\[
\mathcal T^\infty=\{\theta\in\ell_\infty\colon \theta\sim\theta'\text{ if }\theta-\theta'\in\mathbb Z^\mathbb N\}.
\] 
It is a holomorphic Banach manifold (see Lemma \ref{lem:banachmanifold}).
On the other hand, given $\bs m\in \ell_{\infty,\kappa}^\mathbb R$, the natural range of actions in which we will be interested\footnote{ We will see below how this scaling appears naturally when considering finite energy couplings of infinitely many finite dimensional Hamiltonians.} correspond to sequences in the Banach space\black
\[
\ell_{\infty,\bs m}=\{I\in\ell_\infty
\colon |I|_{\infty,\bs m}:=\sup_{i\in\mathbb N} m_i^{-1} |I_i|<\infty\}.
\]
Then, given any two real constants $\sigma,\rho>0$ we define the bounded subsets
\begin{equation}
\label{ComplexNbhd}
\mathcal T^\infty_\sigma=\{\theta\in\mathcal T^\infty\colon |\operatorname{Im}\theta|_\infty\leq \sigma\}\qquad\qquad B_{\bs m,\rho}=\{I\in\ell_{\infty,\bs m} \colon |I|_{\infty,\bs m}\leq \rho\}.
\end{equation}
Having introduced this functional setting, we can now give a precise meaning to the concept of long range interaction.

\begin{defn}[Long range interactions with uniformly bounded potentials]\label{defn:longrange}
    Given a real-analytic function $P:\mathcal T^\infty_\sigma\times B_{\bs m,\rho}\to \mathbb C$ we say it is of \textit{long range} if it is of the form 
      \begin{equation}\label{eq:shortrange}
P(\theta,I)=\sum_{i<j} m_im_j P_{i,j}(\theta_i,\theta_j,I_i,I_j)
\end{equation}
 with  $P_{i,j}$ real-analytic and uniformly bounded
 
 \[
 \sup\{|P_{i,j}(\theta_i,\theta_j,I_i,I_j)|\colon |
 \operatorname{Im}\theta_\star|\leq\sigma,\  |m_\star^{-1} I_\star|\leq \rho,\ \text{for }\star=i,j \}\leq 1.
 \]
\end{defn}
Indeed, if $P$ is of long range we notice from Hamilton's equations that the vector field induced by a Hamiltonian of the form 
$H_\xi(\varphi,J)=\langle \xi,J\rangle +\varepsilon P$  (with respect to the canonical symplectic form $\sum_{i} d \theta_i\wedge dI_i$ on $\mathcal T^\infty\times\ell_{\infty,\bs m}$) reads
\begin{align*}
\dot \theta_i-\xi_i=&\partial_{I_i} P=\varepsilon\sum_{j\neq i} m_im_j \partial_{I_i}P_{i,j}=O(\varepsilon)\\
\dot I_i=&-\partial_{\theta_i} P=-\varepsilon \sum_{j\neq i} m_im_j \partial_{\theta_i}P_{i,j}=O(\varepsilon m_i).
\end{align*}
In particular, in the scaled variables $\widetilde I\in\ell_\infty$ defined by $I=\bs m \widetilde I$, the vector field does not decay as $i\to \infty$. Moreover, and what is more important, for frequencies $\xi\in \ell_\infty$, the ratio 
\begin{equation}\label{eq:freqpertratio}
\frac{(\text{peturbation})_i}{(\text{unperturbed frequency})_i}\simeq O(\varepsilon)
\end{equation}
as $i\to \infty$. Thus for fixed $\varepsilon$, the ratio between the time scale of the unperturbed motion and the time scale of the dynamics induced by the vector field associated to the perturbation $P$ does not necessarily decay as $i\to \infty$.

\begin{rem}
The fact that \eqref{eq:shortrange} only includes binary interactions is not important. Indeed, our abstract theorems in Section \ref{sec:mainabstract} hold for a broader class of perturbations which we introduce in Section \ref{sec:definitions}.

A more thorough and general discussion on long range perturbations and the  the challenges that one faces when developing a KAM theory in that setting, can be found in Section \ref{sec:comparison}.  
\end{rem}

We are now ready to present a non-technical version of our main result. It concerns the persistence of infinite-dimensional invariant tori for perturbations of linear integrable Hamiltonians with external parameters. A more general and precise formulation of this theorem is provided in Section \ref{sec:mainabstract}. When comparing the informal statement in this section with the rigorous one, the reader should keep in mind that the latter is formulated in the naturally scaled phase space, where the domains of analyticity are uniform in $i \in \mathbb{N}$ (see Section \ref{sec:mainabstract}).

 Let $a<b$ be two fixed real numbers and fix any positive constants $\sigma,\rho>0$. Let 
  $\xi\in [a,b]^{\mathbb N}$,  and for small $\varepsilon>0$ consider a real-analytic Hamiltonian $H_\xi:\mathcal T^\infty_\sigma\times B_{\bs m,\rho}\to \mathbb C$ of the form 
\begin{equation}\label{eq:fulldimintrolag0}
H_\xi(\varphi,J)=\langle \xi,J\rangle+\varepsilon P(\theta,I)
\end{equation}
with $P$ as in \eqref{eq:shortrange}. The following is an informal presentation of our  main result.
\begin{thmy}[Full-dimensional KAM tori]\label{thm:introfull}
Let $\bs m\in \ell^\mathbb R_{\infty,\kappa}$ for some $\kappa>0$, and fix any $\sigma,\rho>0$.  
Let 
$$H_\xi:\mathcal T_\sigma^\infty\times B_{\bs m,\rho}\to \mathbb C$$ be as in \eqref{eq:fulldimintrolag0}.  
Then there exists $\varepsilon_0(\kappa,\rho,\sigma,|b-a|)>0$ such that, for any $0\leq \varepsilon\leq \varepsilon_0$, there exist uncountably many $\xi\in [a,b]^\mathbb N$ for which the corresponding Hamiltonian $H_\xi$ possesses a real-analytic\footnote{The precise meaning of real-analytic in this context will be clarified later. For now, it suffices to note that the torus is given by an embedding 
\[
\Phi:\mathcal T_\sigma^\infty \to \mathcal T_\sigma^\infty\times B_{\bs m,\rho}
\]
such that, for any bounded $A\subset \mathbb N$, the dependence $\theta_A\mapsto \Phi$ is real-analytic,  with the analyticity domain in each coordinate $\theta_j$ having a width that is uniform in $j$.} 
full-dimensional invariant torus on which the flow is conjugated to a linear translation on $\mathbb T^\mathbb N$ with frequency $\xi'$, where $\xi'$ is $O_{\ell_\infty}(\varepsilon^{1/2})$-close to $\xi$.  
\end{thmy}
\medskip

A more technical statement is given in Theorem \ref{thm:mainlagrangianprecise}. We emphasize that we do not claim the existence of an invariant torus for a set of large measure of  $\xi\in[a,b]^\mathbb N$ with respect to the product measure. As we discuss in Section \ref{sec:stochlayer}, the long range nature of the interaction in $P$ prevents such a conclusion from holding. A precise description of the \textit{singular} set of frequencies for which we construct  full-dimensional tori is provided in Theorem~\ref{thm:measestimatesgeneral}. In plain words, the set of KAM frequencies in Theorem \ref{thm:introfull} is abundant on a suitable Hilbert cube around any $\xi\in[a,b]^\mathbb N$ for which $\mathrm{dim}_{\mathrm{box}}(\{\xi_i\}_i)<1$.

\subsection{Infinite couplings of non-degenerate, finite-dimensional Hamiltonians}

A natural question which motivates this work is the existence of  finite  energy \black\footnote{For high dimensional Hamiltonians at high energy regimes, Boltzmann's ergodic hypothesis conjectures that the motion is ergodic with respect to Liouville measure.} quasiperiodic motions for  systems composed by an infinite number of particles interacting via a weak all-to-all coupling. 

This problem can be seen as a particular case of a more general one, which we now present (concrete applications to  explicit systems of interacting particles will be given in Section \ref{sec:particles}).  Namely, the construction of KAM tori for infinite couplings of finite-dimensional Hamiltonians. Of course, in order to obtain meaningful results in this rather general context, one has to make some assumptions on the way these systems are coupled. Our results below apply to the following class of infinite-dimensional Hamiltonians.

\begin{defn}\label{defn:Nbodytype1}
Let $N\in\mathbb N$ and $\bs m\in\ell^\mathbb R_{\infty,\kappa}$ for some $\kappa>0$. We say that 
\[
H(x,y):\ell_{\infty}^N\times\ell_{\infty,\bs m}^N\to\mathbb C
\]
is of \textit{class}-$\mathcal H$ if it is of the form
\begin{equation}\label{eq:firstHamnew}
H(x,y)=H_0(x,y)+\varepsilon H_1(x,y)
\end{equation}
with
\begin{equation}\label{eq:firstHamnew2}
H_0(x,y)=\sum_{i=1}^\infty  m_ih_i(x_i,y_i/m_i)
\qquad H_1(x,y)=\sum_{i<j} m_im_j V_{i,j}(x_i,x_j,y_i/m_i,y_j/m_j),
\end{equation}
and there exists $K>0$ such that, for any $i,j\in\mathbb N$,  the functions $h_i\in C^\omega((-K,K)^{N},\mathbb R)$ and 
$V_{i,j}\in C^\omega((-K,K)^{4N},\mathbb R)$ are  uniformly  bounded on some (uniform) compact complex neighbourhood of $(-K,K)^N$ or $(-K,K)^{4N}$. Here, the Hamiltonian dynamics are considered with respect to the canonical symplectic form $\sum_{i} d x_i\wedge dy_i$ on $\ell_{\infty}^N\times\ell_{\infty,\bs m}^N$, $x_i$ and $y_i$ being the position and momentum of ``particle'' $i$. The scaling of $y_i \sim m_i$ corresponds to a finite total energy regime.
\end{defn}
\medskip

Indeed, as we will see in Section \ref{sec:particles}, systems of the form \eqref{eq:firstHamnew}-\eqref{eq:firstHamnew2} appear naturally when studying systems of interacting particles with bounded energy. In that case, for all $i\in\mathbb N$, 
\[
h_i(x_i,y_i)=\frac{|y_i|^2}{2}-V(x_i)\qquad\qquad (\text{hence}\quad m_i h_i(x_i,y_i/m_i)=\frac{|y_i|^2}{2m_i}-m_iV(x) \text{ is the mechanical energy})
\]
for some background  potential $V(x)$. One may think of \eqref{eq:firstHamnew}-\eqref{eq:firstHamnew2} as a \textit{bounded energy},  infinite coupling of finite dimensional Hamiltonians. 
\medskip

\begin{rem}[Long range interactions for class-$\mathcal H$ Hamiltonians]
    Consider a class-$\mathcal H$ Hamiltonian as in Definition \ref{defn:Nbodytype1}. The equations of motion for such a system read
    \begin{align*}
    \dot x_i=&\partial_{y_i} h_i (x_i,y_i/m_i)+\varepsilon\sum_{i\neq j}m_j\partial_{y_i}V_{i,j}(x_i,x_j,y_i/m_i,y_j/m_j)=O(1)\\
    \dot y_i=&-m_i\partial_{x_i} h_i (x_i,y_i/m_i)-\varepsilon\sum_{i\neq j}m_i m_j\partial_{y_i}V_{i,j}(x_i,x_j,y_i/m_i,y_j/m_j)=O(m_i)
    \end{align*}
    In particular, in the (naturally) scaled variables $\tilde y\in\ell_\infty^N$ defined by $y=\bs m\tilde y$, the vector field does not decay as $i\to \infty$ (nor does the ratio between the time scale of the unperturbed motion and that of the dynamics induced by the coupling term).
\end{rem}

We turn now to the hypotheses we assume on the $N$ degrees-of-freedom Hamiltonians 
\[
h_i(x_i,y_i)\in\mathbb R^N\times\mathbb R^N\to \mathbb R,
\]
 in our study of a Hamiltonian $H$ as in \eqref{eq:firstHamnew}-\eqref{eq:firstHamnew2}. We will suppose that these finite dimensional Hamiltonians satisfy one of the following properties:
\begin{itemize}
\item \textbf{P1} \textit{(non-degenerate Lagrangian invariant torus):} There exist constants $\rho,\sigma,\gamma>0$, an exponent $\tau\geq N$  and a compact set $K\subset\mathbb R^N$ such that, for all $i\in\mathbb N$: there exist  a subset $M_i\simeq \mathbb T^N_\sigma\times B_\rho^N$ of the phase space, and a symplectic, real-analytic local coordinate system $(\varphi,J)$ on $M_i$ such that $h_{i}$ can be written in normal form
\begin{equation}\label{eq:normalformP1}
N_i(\varphi,J)=\langle \xi_0^i, J\rangle+\frac12\langle A_i J,J\rangle +P_i(\varphi,J),
\end{equation}
for some  $\xi_0^i\in K\subset \mathbb R^N$ in the Diophantine class $DC(\gamma,\tau)$, with $A_i$ an invertible matrix, and $P_i$ satisfying  
$\partial_{J^n}^{n}P_i|_{\{J=0\}}=0$ for all $n\leq 2$. We moreover assume that 
$\displaystyle \sup_i \lVert A_i^{-1}\rVert<\infty$. 

\item \textbf{P2} \textit{(non-degenerate elliptic equilibrium):} There exist constants $\rho,\gamma>0$ and a compact set $K\subset\mathbb R^N$ such that, for all $i\in\mathbb N$: there exists a subset $M_i\simeq  \mathbb B^{2N}_\rho$ (here $\mathbb B_\rho\ni 0$ is a complex disk in $\mathbb C$ of radius $\rho$) of the phase space and a symplectic, real-analytic local coordinate system $(z,\bar z)$ on $M_i$ on which $h_{i}$ can be written in normal form
\begin{equation} \label{eq:normalformP2}
N_{i}(z,\bar z)=\langle\omega_0^i,z\bar z\rangle 
+ \frac12 \langle A_i z\bar z,z\bar z\rangle +P_i(z,\bar z),
\end{equation}
for $\omega_0^i\in K\subset \mathbb R^{N}$, satisfying 
\[
|\omega_0^i\cdot k |\geq \gamma>0
\]
for all $k\in\mathbb Z^N$ with $0<|k|_1\leq 100$, with $A_i$ being an invertible matrix, and 
$\partial_z^{n_z}\partial_{\bar z}^{n_{\bar z}} P|_{\{z=\bar z=0\}}=0$ whenever $n_z+n_{\bar z}\leq 2$.  
We moreover assume that  
$\displaystyle \sup_i \lVert A_i^{-1}\rVert<\infty$.
\end{itemize}

In the study of coupled Hamiltonian systems, it is natural to look at regions of the phase space where each system displays a non degenerate elliptic equilibrium or quasi-periodic invariant torus.  An instance of this approach in the recent literature can be found in \cite{TuraevRomkedar} (see also \cite{MR2383597} and the
references therein), where the authors construct KAM islands around (normally-elliptic) choreography-type  periodic orbits for systems of interacting particles.  \black 

\subsection*{Full-dimensional tori}

For $\varepsilon=0$, a class-$\mathcal H$ Hamiltonian $H$ as in \eqref{eq:firstHamnew} satisfying property \textbf{P1} reduces to an uncoupled sum of infinitely many Hamiltonians, each displaying an $N$-dimensional non-degenerate real-analytic invariant torus. The infinite-dimensional Hamiltonian $H_0$ ($H$ of \eqref{eq:firstHamnew} with $\varepsilon=0$) has an infinite full-dimensional invariant torus $\mathcal T_0$ consisting of the infinite Cartesian product of these uncoupled tori. Moreover, the differential of the amplitude-frequency map for the infinite-dimensional system \eqref{eq:firstHamnew} is given by a block-diagonal matrix whose blocks are invertible. Looking at the higher order terms $P_i$ of each finite dimensional Hamiltonian  in \eqref{eq:normalformP1} as perturbative terms, we see that $H_0$ has a quasi-integrable structure around $\mathcal T_0$ given by the infinite Cartesian products of the uncoupled tori of the quadratic (in action) part of the Hamiltonians $N_i$. We refer to these tori as \textit{uncoupled tori}.

The next  result shows that, for $\varepsilon>0$ sufficiently small, many of the uncoupled tori persist as slightly deformed full-dimensional invariant tori for the Hamiltonian \eqref{eq:firstHamnew}.

\begin{thmx}[Full-dimensional KAM tori]\label{thm:mainLagrangianinformal}
Let $N\in\mathbb N$, let $\bs m\in\ell^\mathbb R_{\infty,\kappa}$ for some $\kappa>0$, and let $H:\ell_\infty^N\times\ell_{\infty,\bs m}^N\to\mathbb C$ be a Hamiltonian of class $\mathcal H$ for which the corresponding $\{h_i\}_{i\in\mathbb N}$ satisfy \textbf{P1}. Then there exists a constant $\varepsilon_0(\{h\}_i,\{V_{ij}\},\kappa)>0$ such that for any $0\leq \varepsilon\leq \varepsilon_0$ there exists an uncountable collection of real-analytic full-dimensional tori\footnote{In the neighborhood of the infinite invariant torus $\mathcal T_0$ of $H_0$.}, invariant under the flow of $H$, on which the motion is conjugated to a linear translation on $\mathbb T^{N\mathbb N}$.

\end{thmx}

Theorem \ref{thm:mainLagrangianinformal} is obtained in Section \ref{sec:actionanglenormalform} as a consequence of Theorems \ref{thm:mainlagrangianprecise} and \ref{thm:measestimatesgeneral}, which, together, constitute a more precise version of Theorem \ref{thm:introfull}. 
\medskip

A more detailed description of the set of KAM  tori in Theorem \ref{thm:mainLagrangianinformal} is also given in Section \ref{sec:actionanglenormalform}. Roughly speaking, the KAM set in Theorem \ref{thm:mainLagrangianinformal} is abundant with respect to a \textit{singular measure} obtained by the following construction. Choose any initial condition  for which the unperturbed frequency $\xi\in \ell_\infty^N$, when seen as a sequence living in $\mathbb R^N$, has \textit{box dimension} smaller than one. Then, around this initial condition consider a probability measure supported on a Hilbert cube (with suitable polynomial decay in $i$ according to the box dimension of $\{\xi_i\}_{i\in\mathbb N}\subset \mathbb R^N$). It is with respect to such a measure (associated to any choice  of initial condition as above) that the KAM set in Theorem \ref{thm:mainLagrangianinformal} is abundant (see Section \ref{sec:actionanglenormalform} for the precise construction). In this perspective, Theorem \ref{thm:mainLagrangianinformal} gives a natural extension of the classical KAM-stability of a non-degenerate quasi-periodic torus 
 to the case of infinite-dimensional  non-degenerate quasi-periodic tori, with the abundance of quasi-periodic motion obtained with respect to a class of singular measures on the set of actions.

 We argue in Section \ref{sec:stochlayer} that KAM tori cannot be expected to be abundant with respect to 
 product of Lebesgue measures in our infinite dimensional setting as this would necessarily include excessively strong resonances.
\medskip

\subsection*{KAM annuli around elliptic fixed points}

For $\varepsilon=0$, a class-$\mathcal H$ Hamiltonian $H$ as in \eqref{eq:firstHamnew} satisfying property \textbf{P2} reduces to an uncoupled sum of infinitely many Hamiltonians, each displaying an elliptic non-degenerate equilibrium.  The infinite-dimensional Hamiltonian $H_0$  has then an infinite dimensional elliptic equilibrium that we denote by $\mathcal E_0$. Looking at the higher order terms $P_i$ of each finite dimensional Hamiltonian  in \eqref{eq:normalformP2} as perturbative terms, we see as above that $H_0$ has a quasi-integrable structure around $\mathcal E_0$. We again refer to these tori as \textit{uncoupled tori}. As a corollary of Theorem \ref{thm:mainLagrangianinformal} and a standard normal forms argument, our next  result shows that, for $\varepsilon>0$ sufficiently small, many of the uncoupled tori, of a small annulus around $\mathcal E_0$, persist as slightly deformed full-dimensional invariant tori for the Hamiltonian \eqref{eq:firstHamnew}.

\begin{thmx}[KAM annuli]\label{thm:mainellipticpts}
Let $N\in\mathbb N$, let $\bs m\in\ell^\mathbb R_{\infty,\kappa}$ for some $\kappa>0$, and let $H:\ell_\infty^N\times\ell_{\infty,\bs m}^N\to\mathbb C$ be a  Hamiltonian of class $\mathcal H$ for which the corresponding $\{h_i\}_{i\in\mathbb N}$ satisfy property \textbf{P2}. Then there exists a constant $\varepsilon_0(\{h_i\},\{V_{ij}\},\kappa)>0$ such that for any $0<\varepsilon\leq \varepsilon_0$ there exists an uncountable collection of full-dimensional tori, invariant under the flow of $H$, on which the motion is conjugated to a linear translation on $\mathbb T^{N\mathbb N}$. These tori are localized within a small annulus  around $\mathcal E_0$ \black in $\ell_\infty^N\times\ell_{\infty,\bs m}^N$.
\end{thmx}


The abundance in the KAM-annulus around $\mathcal E_0$ of quasi-periodic motion is again obtained with respect to a class of singular measures on the set of actions.
We do not claim that the KAM-annulus that we display accumulates  on the continuation of the elliptic fixed point. Establishing such a result would require additional hypotheses on the frequencies at the elliptic fixed point. To avoid unnecessary length, we do not pursue this direction here.
\medskip

\subsection{The finite energy regime: motivation from celestial mechanics}

A rather natural instance of high-dimensional dynamics in finite energy regimes might correspond to the \textit{asteroid belt} in our solar system. Below we describe a \textit{very simplified} picture of this scenario.

Consider a system formed by an infinite number of bodies with masses $\mu_0=1$ and $\mu_i=\varepsilon e^{-i}$ for $i=1,\dots,\infty$ and some $0<\varepsilon\ll 1$. Suppose that the bodies interact 
via a \textit{smooth potential} which is proportional to the product of their masses and their distance.  As we will see in Section \ref{sec:particles},  the Hamiltonian governing the dynamics of this system is given by a Hamiltonian of the form \eqref{eq:firstHamnew}--\eqref{eq:firstHamnew2}. In a large part of the phase space, the dynamics corresponds (for $\varepsilon=0$) to motions in which each of the lighter particles $i=1,\dots, \infty$ (planets and asteroids) revolve around the heavy body $i=0$ (the sun) in a quasiperiodic fashion. For these configurations, the total energy of the system is finite (see Section \ref{sec:particles}).

Within this framework (assuming moreover that the smooth interaction potential satisfies mild nondegeneracy assumptions), we will construct in Section \ref{sec:particles} a  set of initial conditions for which the coupled system (i.e. with $\varepsilon\neq 0$) displays almost-periodic dynamics. The singular measure for which these KAM tori are abundant admits a rather natural interpretation in this scenario which might connect to the distribution of the so-called \textit{Kirkwood gaps} in the asteroid belt (see \cite{MR3551192} and the references therein). Namely, the lighter bodies i.e. those with $i\gg 1$, tend to accumulate on regions for which their corresponding unperturbed frequency is strongly non-resonant with respect to that of the bodies of large mass (i.e. those with $i=O(1)$) (see Figure \ref{fig:placeholder}).
\black

\begin{figure}
    \centering
    \includegraphics[scale=0.4]{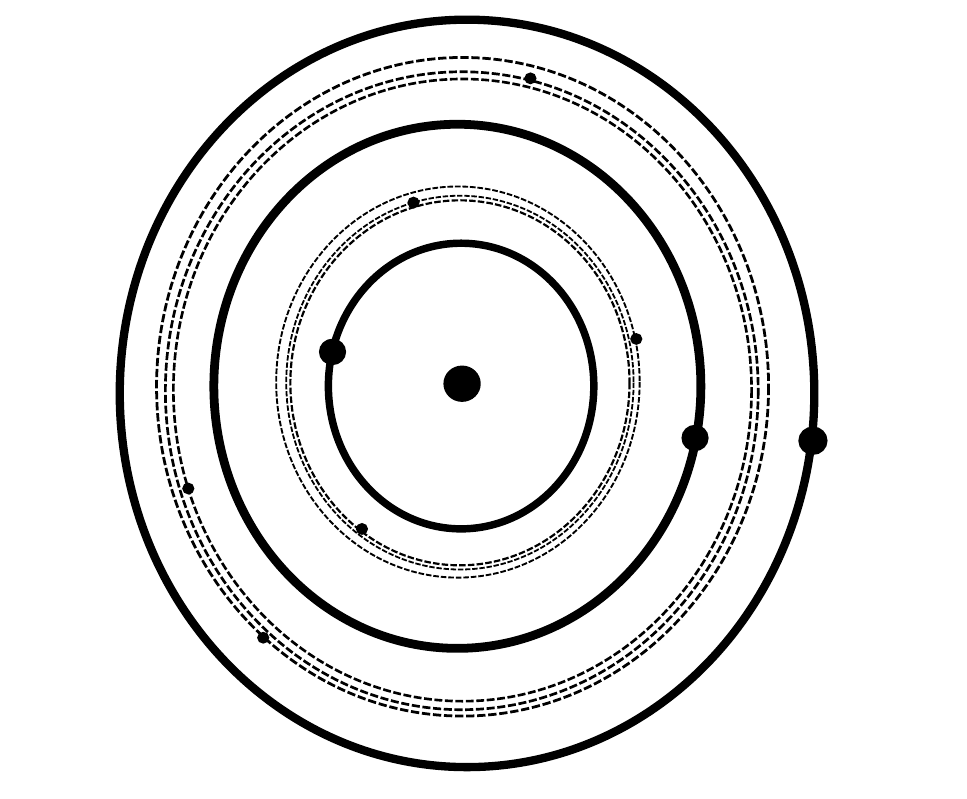}
    \caption{The smaller bodies tend to accumulate on regions for which their interaction with the large bodies is non-resonant. In this schematic picture the dashed curves correspond to approximate orbits of light particles while the solid ones correspond to approximate orbits of heavy bodies.}
    \label{fig:placeholder}
\end{figure}
\subsection{Organization of the article}

The structure of this article is organized into five main blocks:

\begin{itemize}
    \item \textbf{Normal Forms.} 
    In Sections~\ref{sec:mainabstract}, \ref{sec:functionalsetting}, \ref{sec:Kamlagrangian}, \ref{sec:measureestimatesLag}, we introduce and analyze the dynamics of suitable normal forms which allow us to:
    \begin{itemize}
        \item formulate the precise statement of our main theorem (Section~\ref{sec:mainabstract});
        \item establish the functional framework required for the KAM iteration (Section~\ref{sec:functionalsetting});
        \item apply KAM techniques to construct full-dimensional invariant tori (Sections~\ref{sec:Kamlagrangian}--\ref{sec:measureestimatesLag}).
    \end{itemize}

\item \textbf{Non-KAM set} In Section \ref{sec:stochlayer} we provide an argument which suggests that, for the class of systems considered in this paper, the KAM set cannot have positive measure with respect to the standard product measure.

    \item \textbf{Comparison.} 
    In Section~\ref{sec:comparison} we compare our abstract results with related works in the existing literature and emphasize the main novelties of our approach.

    \item \textbf{Infinite couplings.} In Section \ref{sec:actionanglenormalform} we complete the proofs of Theorems \ref{thm:mainLagrangianinformal} and \ref{thm:mainellipticpts}. The argument boils down to obtaining a (infinite-dimensional) normal form for which the results in Section \ref{sec:mainabstract} apply.

    \item \textbf{Applications.} 
  In Section~\ref{sec:particles} we present applications of our techniques to the construction of almost-periodic motion in systems of infinitely many interacting particles.
\end{itemize}

\section{Notation, phase space and parameters}\label{sec:definitions}

In this section we discuss some concepts which are needed to state more precise formulations of our main results. These are stated for perturbations of the abstract normal form \eqref{eq:normalformLag}.

\subsection{Notations}
 The following notations will be used in what follows. 
\begin{itemize}
\item We define the sets (given $A\subset\mathbb N$ we denote by $|A|$ its cardinality)
\[
\mathbb N_0=\{A\subset \mathbb{N} \colon |A|<\infty\}\subset 2^\mathbb N
\]
and
\begin{align*}
\mathbb N_\infty=&\{\alpha\in  \mathbb N^{\mathbb N}\colon \alpha_n\neq 0\text{ only for finitely many $n$}\}\\
\mathbb Z_\infty=&\{l\in \mathbb Z^{\mathbb N}\colon l_n\neq 0\text{ only for finitely many $n$}\}\\
\end{align*}
 \item Given $A\in\mathbb N_0$
  we denote by 
    \[
    \overline A=\max\{j\in A\},\qquad\qquad \underline A=A\setminus \{\overline A\}.
    \]
    
    \item Given $A\in \mathbb N_0$, $l\in \mathbb{Z}_\infty$ and  $\alpha\in\mathbb{N}_\infty$, we define (by $|l|$ we mean the vector $|l|=(|l_1|,|l_2|,\dots)$, do not confuse with the $|l|_1$ norm)
\[
\mathcal{A}=\{(l,\alpha)\in\mathbb{Z}_\infty\times\mathbb{N}_\infty\times\mathbb{N}_\infty\times\mathbb{N}_\infty\colon \mathrm{supp}(|l|+\alpha)=A\}.
\]

\item  For $\bs m\in\ell^\mathbb R _{\infty,\kappa}$ and any  $A\in\mathbb N_0$ we define
\[
m_A=\prod_{j\in A}m_j.
\]
\end{itemize}

\subsection{Phase space}
In this section we introduce the phase space in which the abstract normal form \eqref{eq:normalformLag} is defined. 
We let
\begin{equation}
\label{PhaseSpace}
\mathcal Q=\{(\varphi,J)\in \mathcal T_\infty\times\ell_\infty\}
\end{equation}
and, given constants $\rho>0$ and $\sigma>0$ define the bounded subset  (cf. \eqref{ComplexNbhd})
\begin{equation}\label{eq:phasespacefulldim}
\mathcal Q_{\rho,\sigma}=\left\{(\varphi,J)\in\mathcal Q\colon\max\left\{\frac{1}{\sigma}|\mathrm{Im}\varphi|_\infty, \frac{1}{\rho}|J|_\infty\right\} < 1\right\}\subset \mathcal Q.
\end{equation}
\\

Given two Banach spaces $E,F$ and an open subset $U\subset E$ we will denote the space of bounded holomorphic mappings from $U$ to $F$ by  \begin{equation}\label{eq:Banachspacehol}
\mathcal H(U,F)=\{\Phi:U\to F\colon  \text{holomorphic and }\sup_{a\in \overline U}\lVert \Phi(a)\rVert_F<\infty\}.
\end{equation}
This is a Banach space (see Lemma \ref{lem:banachspaceholom} in Appendix \ref{sec:proofBanach}).

\subsubsection{Dynamics on scaled symplectic Banach manifolds}\label{sec:dynsympBanachmfolds}
We now introduce a symplectic structure on $\mathcal Q$ and $\mathcal M$. In the following we fix $\kappa>0$ and $\boldsymbol m\in\ell^\mathbb R_{\infty,\kappa}$ (which was defined in \eqref{eq:spacedecayingmasses}). Then, we  
equip $\mathcal Q$ with the symplectic form $\lambda: T\mathcal Q\times T\mathcal Q\to \mathbb C$ which, in local coordinates $(\varphi,J)\in\mathcal Q$ reads
\begin{equation}
\label{Symplectic}
\lambda=\sum_{n\in\mathbb N}m_n\  \mathrm{d}J_n\wedge\mathrm{d}\varphi_n.
\end{equation}
In what follows, we will consider Hamiltonians defined on open subsets of the symplectic manifold $(\mathcal Q,\lambda)$. Motivated by the application to systems of interacting particles, given $\rho>0$ and $\sigma>0$ we consider Hamiltonians belonging to a closed  subspace of $\mathcal H(\mathcal Q_{\rho,\sigma},\mathbb C)$  which we define as follows.  Given $h\in\mathcal H(\mathcal Q_{\rho,\sigma},\mathbb C)$ there exists a unique way\footnote{Uniqueness follows from the uniqueness of Taylor-Fourier series expansions for functions in $\mathcal H(\mathcal Q_{\rho,\sigma},\mathbb C)$.} of writing
\[
h=\sum_{A\in\mathbb N_0} h_A
\]
where 
\[
h_A=\sum_{(l,\alpha)\in\mathcal A}h^{[l]}_{\alpha} \ J^\alpha e(l\varphi)
\]
Then, for $\beta\in (0,1)$ we define the Banach space
\begin{equation}\label{eq:defnBanachspace}
\mathcal Y_{ \beta ,\rho,\sigma}=\{h\in\mathcal H(\mathcal Q_{\rho,\sigma},\mathbb C)\colon \ \lVert h\rVert_{\beta,\rho,\sigma}<\infty\}
\end{equation}
where
\begin{equation}\label{eq:2normLag}
\lVert h\rVert_{\beta,\rho,\sigma}=\sup_{j\in\mathbb{N}}\ \frac{1}{m_j}\sum_{\overline A=  j} (m_{\underline A})^{-\beta} |h_{A}|_{\rho,\sigma},
\end{equation}
and
\[
|h_A|_{\rho,\sigma}=\sum_{(\alpha,l)\in\mathcal{A}} |h^{[l]}_{\alpha}|\  \rho^{|\alpha|_1} \ e^{|l|_1\sigma}.
\]

\subsubsection{Hamiltonian flows on $\mathcal Q$} Given $H\in\mathcal Y_{\beta,\rho,\sigma}$ we can associate to it the unique vector field $X_H:\mathcal Q_{\rho/2,\sigma/2}\to T_{\mathcal Q_{\rho/2,\sigma/2}} \mathcal Q_{\rho,\sigma}$ which satisfies
\[
\lambda(X_H,v)=\mathrm{d}H(v)\qquad\qquad \forall v\in T \mathcal Q_{\rho,\sigma}\simeq \ell_\infty^2.
\]
In local coordinates, these equations read
    \[
    X_H=( (X_H)_1,(X_H)_2,\dots, (X_H)_n,\dots),\qquad\qquad 
    (X_H)_n= m_n^{-1} (\partial_{J_n}H,-\partial_{\varphi_n}H).
    \]
Under these circumstances, the classical existence and uniqueness theorem for ODEs (see, for instance, \cite{DiffmfoldsLang}) yields the existence of a local flow for $H$: that is, given any open subset 
$V\subset \mathcal Q_{\rho/2,\sigma/2}$, there exists a non-empty open interval $I\subset \mathbb R$ centered at $\{0\}$ and a real-analytic map $\Phi_H:V\times I\to \mathcal Q_{\rho,\sigma}$ which, for every $(u,t)\in V\times I$, verifies
\[
\partial_t \Phi_H(u,t)= X_H(\Phi_H(u,t)).
\]
One thus may wonder whether there exist invariant submanifolds for the flow $\Phi_H$ on which, in particular, the flow is globally defined (for example, invariant submanifolds entirely contained in $\mathcal Q_{\rho,\sigma}$).  We will be particularly interested in studying the existence of invariant tori. More concretely, we will seek conjugacies for the flow $\Phi_H$ such that the conjugated flow has  invariant submanifolds of the form $\mathcal T_{\infty}\times\{0\}$ on which it is conjugated to a linear translation $t\mapsto \theta+\omega t$ with rotation vector $\omega\in\ell_\infty$.\black   We should point out, that, in contrast with the finite dimensional situation, such flows are never transitive (as this would imply the existence of a countable dense subset for $\ell_\infty$, but $\ell_\infty$ is not separable).

\subsection{External parameters}
The perturbed normal form \eqref{eq:normalformLag}  below  
 depends on external parameters $\xi\in\ell_\infty$. To account for the dependence on these parameters we introduce some extra notation. Namely, given any (complex) open subset $O\subset\ell_\infty$  we denote by 
\begin{equation}\label{eq:defnBanachspaceLag}
\mathcal Y_{\beta,\rho,\sigma,O}=\{h:\mathcal Q_{\sigma,\rho}\times O\to\mathbb{C}\colon h\ \text{is real-analytic},\ \lVert h\rVert_{\beta,\rho,\sigma,O}:=\sup_{\xi\in \overline O} \lVert h(\cdot,\xi)\rVert_{\beta,\rho,\sigma}<\infty\} 
\end{equation}
Clearly $\mathcal Y_{\beta,\rho,\sigma,O}\subset\mathcal H(\mathcal Q_{\rho,\sigma}\times O, \mathbb C)$.

\section{Main results for abstract normal forms}\label{sec:mainabstract}

For notational purposes, it will be convenient to define the \textit{mass product} on $\ell_\infty$ by
\begin{equation}Â \label{def.massproduct}Â 
a\cdot b=\sum_{i\in\mathbb N} m_i a_{i}b_{i}
\end{equation} 
  Recall that we  always assume that  
\[
\bs m\in \ell^\mathbb R_{\infty,\kappa}=\{\boldsymbol m=\{m_i\}_{i\in\mathbb N}\subset (0,\infty) \colon \sup_{i\in\mathbb N} m_i\ \exp(\kappa i)<\infty\}
\]
for some fixed $\kappa>0$ and that $\{m_i\}_{i\in\mathbb N}$ forms a decreasing sequence.  We let $\beta\in(0,1)$, let $\rho,\sigma>0$, and, for $a<b$, denote by $O \subset \ell_\infty$ the complex $\rho$-neighbourhood of $[a,b]^{\mathbb N}$. Then,  we consider a Hamiltonian $H_{\xi}$ given by
\begin{equation}\label{eq:normalformLag}
H_{\xi}(\varphi,J)=\xi\cdot J+P(\varphi,J;\xi)
\end{equation}
where $P=P_h+\varepsilon \widetilde P$ for some  $P_h, \widetilde P\in\mathcal Y_{\beta,\rho,\sigma,O}$ satisfying:
\begin{itemize}
    \item \textit{(Perturbation part\black):} $\lVert \widetilde P\rVert_{\beta,\rho,\sigma,O}\leq 1$,
\item \textit{(Normal form part\black):}  
$\lVert P_h\rVert_{\beta,\rho,\sigma,O}\leq 1$ and
\[
P_h|_{\{J=0\}}=\partial_J^n P_h|_{\{J=0\}}=0
\]
 for any $n\in\mathbb N^{\mathbb N}$ with $|n|_1= 1$. 
\end{itemize}
\color{black}In words, we consider a $O(\varepsilon)$ perturbation of a system ($\xi\cdot J+P_h$) which has an invariant torus $\{J=0\}$  that is almost periodic  with frequency $\xi$. Our purpose is to investigate the fate of this invariant torus for $\varepsilon>0$.

\begin{rem}[Long range interaction]
    It is worth pointing out that the equations of motion associated to the Hamiltonian \eqref{eq:normalformLag} on the symplectic manifold $(\lambda,\mathcal Q)$   (defined by \eqref{PhaseSpace} and \eqref{Symplectic})
    are given by 
    \begin{equation}\label{eq:vectorfieldnormalform}
        \begin{split}
            \dot \varphi_n=\xi_n-\frac{1}{m_n}\partial_{J_n}P\qquad\qquad \dot J_n=-  \frac{1}{m_n} \partial_{\varphi_n} P
        \end{split}
    \end{equation}
    what reflects the long range nature of the problem. Indeed, suppose that  
    $\displaystyle\widetilde P=\sum_{i<j}m_im_j \widetilde P_{ij}$ with $|\widetilde P_{ij}|_{\rho,\sigma}\leq 1$ so $\widetilde P\in\mathcal X_{\beta,\rho,\sigma}$ for any $\beta<1$. The equations \eqref{eq:vectorfieldnormalform} for $\varepsilon  \widetilde P$ read 
\begin{align*}
\dot \varphi_n=\xi_n-\varepsilon \sum_{i\neq n}m_i\partial_{J_n}\tP_{in}\qquad\qquad \dot J_n=-  \varepsilon  \sum_{i\neq n}m_i \partial_{\varphi_n} \tP_{in}
\end{align*}
We observe that the mass $m_n$ disappears, and the $n$-th component of the vector field can only be bounded by $\varepsilon$ for all $n\in\mathbb{N}$. This corresponds precisely to the long range setting discussed in 
Section~\ref{sec:intro}. For systems of interacting particles (introduced in Section \ref{sec:particles}), the long range nature of the interaction is already apparent in Newton's equations \eqref{eq:newton}.  
\end{rem}
\vspace{0.2cm}
\begin{thm}[Full-dimensional KAM tori]\label{thm:mainlagrangianprecise}
Let $\beta \in (0,1)$, let $\rho,\sigma>0$, and let $H_{\xi}\in\mathcal Y_{\beta,\rho,\sigma,O}$ be the Hamiltonian \eqref{eq:normalformLag} on the symplectic manifold $(\lambda,\mathcal Q)$. Then, for $\varepsilon>0$ sufficiently small\footnote{A (very far from optimal) condition  of the form $\varepsilon\lesssim \min\{\rho^6\sigma^6,\ |b-a|^{24}\}$ suffices.}, there exist a  non-empty real subset $\mathcal O_F\subset O$ and a real-valued Lipschitz\footnote{The map is Lipschitz when viewed as a map from a subset of $\ell_\infty^2$ to $\ell_\infty^2$.} map
\[
\phi:\mathcal O_F \to O
\]
such that, for any $\xi\in \mathcal O_F$, there exists a real-analytic symplectic map $\Phi \in \mathcal H(\mathcal Q_{\rho/2,\sigma/2}, \mathcal Q_{\rho,\sigma})$, $\mathcal O(\varepsilon^{1/2})$-close to the identity, such that 
\[
\Phi(\mathcal Q_{\rho/2,\sigma/2}) \subset \mathcal Q_{\rho,\sigma},
\]
and
\[
H_{\phi(\xi)}\circ \Phi(\varphi,J) \;=\; \xi\cdot J + R(\varphi,J;\xi),
\]
where
\[
R|_{J=0}=0 
\qquad\text{and}\qquad 
\partial^{n}_J R|_{J=0}=0
\]
for all $n\in\mathbb N^{\mathbb N}$ with $|n|_1=1$.  
\end{thm}

The consequences of the theorem are well known: the real-analytic infinite dimensional torus 
\[
\mathtt T_\infty=\Phi(\{J=0\})\subset\mathcal Q_{\rho,\sigma}
\]
is invariant for the flow of $H_{\phi(\xi)}$ and the flow restricted to it is conjugated to the linear translation $t\mapsto \varphi+\xi t$.  Moreover, since 
$\Phi \in \mathcal H(\mathcal Q_{\rho/2,\sigma/2}, \mathcal Q_{\rho,\sigma})$ is $\mathcal O(\varepsilon^{1/2})$-close to the identity with respect to the metric on $\mathcal Q_{\rho/2,\sigma/2}$ induced by the $\ell_\infty$ norm, it follows that the associated invariant torus $\mathtt T_\infty$ is $\mathcal O(\varepsilon^{1/2})$-close (in the same metric) to the unperturbed invariant torus corresponding to $\varepsilon=0$, namely $\{J=0\}$.  
This means that for any bounded $A \subset \mathbb N$, the projection of $\mathtt T_\infty$ onto the $A$-coordinates is a real-analytic torus that is $\mathcal O(\varepsilon^{1/2})$-close to $\{J_A=0\}$, in analyticity domains of uniform width with respect to $A$.

\subsection{Measure estimates}

The next natural question is to measure the ``size'' of the set $\mathcal O_F\subset [a,b]^{\mathbb N}$ of frequencies for which Theorem \ref{thm:mainlagrangianprecise} holds.  We show below that, provided $\varepsilon$ is small enough, the set $\mathcal O_F$  in Theorem \ref{thm:mainlagrangianprecise} is not empty. Moreover, we show that these sets are abundant in the neighborhood of frequencies with \textit{box dimension} strictly smaller than one.

\begin{thm}\label{thm:measestimatesgeneral}
 Fix any sequence $\tilde{\xi}:=\{\tilde\xi_n\}_{n\in\mathbb N}\subset[a,b]^{\mathbb N}$ with $\mathrm{dim}_{box}(\tilde{\xi})=d\in(0,1)$, and define
\[
\Lambda_d(\tilde{\xi})=\prod_{n=1}^\infty \left[\tilde\xi_{n}-\frac{1}{n^{1/2d}},\;\tilde\xi_{n}+\frac{1}{n^{1/2d}} \right].
\]
Then, if $\mu$ denotes the standard probability measure on the product $\sigma$-algebra of $\Lambda_d(\tilde{\xi})$ and denote by $\mathcal O_F$ the set of frequencies for which Theorem \ref{thm:mainlagrangianprecise} holds, we have 
\[
\mu(\mathcal O_F\cap \Lambda_d(\tilde{\xi}))\;\geq\; 1 - O\!\left(\varepsilon^{\frac{1}{24}(1-d)}\right).
\]
\end{thm}

We observe that for $\mathrm{dim}_{box}(\tilde\xi)=d\in (0,1)$ it is possible to find non-resonant vectors $\xi$ which are close to $\tilde\xi$ in the weighted $\ell_\infty$ norm with weights $n^{1/d}$. In this sense, as $d\to 0$ we can find non-resonant vectors closer to the given $\tilde\xi$.

As will become apparent from the proof of Theorem \ref{thm:measestimatesgeneral}, carried out in Section \ref{sec:measureestimatesLag}, the main idea behind this result is that working
with highly concentrated set of frequences leads to overlaps in resonant zones which allows to reduce
the number of resonant zones to a finite set at each step of the KAM procedure.

 \begin{rem}  From the proof of Theorem \ref{thm:measestimatesgeneral} we observe that if $\mathrm{dim}_{box}(\tilde\xi)=d'<d$ then, also
\[
\mu(\mathcal O_F\cap \Lambda_d(\tilde\xi))\geq 1-O(\varepsilon^{\frac{1}{24}(1-d)}).
\]
 In particular, we can also obtain measure estimates for the case where $\mathrm{dim}_{box}(\tilde\xi)=0$. In this case given $\tilde\xi$ we can find non-resonant vectors $\xi$ which are super-polynomially close to $\tilde\xi$: i.e. they are close in any weighted $\ell_\infty$ norm with weight $n^{1/d}$, $d>0$. 
\end{rem}

\section{Measure of the non-KAM set}\label{sec:stochlayer}
In this section we provide a heuristic  argument which explains why it is not reasonable to expect that Theorem \ref{thm:mainlagrangianprecise} holds for a positive measure set with respect to the product measure on $[a,b]^\mathbb N$. For technical reasons we only present the argument for Hamiltonians on high dimensional manifolds but, after some additional work, it is possible to adapt our construction to the infinite-dimensional setting.
\medskip

We assume that  for some $0<\delta<1$
\begin{equation}\label{eq:fastmassdecay}
m_1=1\qquad\qquad \text{and }\qquad\qquad m_i\leq \delta e^{-i}\qquad\text{ for }i\geq 2.
\end{equation}
Fix any $N\in\mathbb N$ and introduce the mass scalar product for $u,v\in\mathbb R^N$
\[
u\cdot v=\sum_{i\leq N} m_i u_iv_i.
\] 
Then, for $\xi\in[1,2]^N$ we consider a real-analytic  Hamiltonian $H_\xi(\varphi,J):\mathbb T_\sigma^N\times B^N_\rho\to\mathbb C$ (here $B^N_\rho\subset\mathbb R^N$ is the ball of radius $\rho$ centered around the origin) of the form 
\begin{equation}\label{eq:findim}
H_\xi=\xi\cdot J+\frac 12 J\cdot J+\varepsilon\sum_{i\leq N}m_1m_i\cos(\varphi_1-\varphi_i)
\end{equation}
and study the flow it generates with respect to the symplectic form 
\[
\lambda=\sum_{i\leq N}m_i\mathrm dJ_i\wedge  \mathrm d\varphi_i.
\]
More concretely, we want to show the following. Let $\varrho>0$ and for any $j\in\mathbb N$ define the \textit{resonant parameter set}
\[
R_j(\varrho)=\{\xi\in [1,2]^N\colon |\xi_1-\xi_j|\leq \varrho\sqrt\varepsilon\}.
\]
\textbf{Claim:} \textit{Fix any $\varrho>0$ sufficiently small (independent of $N$) and suppose that $\bs m=(m_1,\dots,m_N)$ is as in \eqref{eq:fastmassdecay} with $0\leq \delta\ll \varepsilon$.}  

\textit{Then, for any $2\leq j\leq N$, if $\xi\in R_j(\varrho)$, the Hamiltonian \eqref{eq:findim} does not possess any  \textit{Lagrangian} invariant torus in the region $\{|I|\leq C\sqrt\varepsilon\}$ for some $C>0$ which is independent of $j$ and $N$.}
\medskip

From this claim it is easy to conclude that the set of parameters 
\[
\mathcal R=\{\xi\in[1,2]^N\colon H_\xi \text{ does not exhibit a Lagrangian invariant torus in } |I|\leq  C\sqrt\varepsilon\}
\]
has Lebesgue measure 
\[
\mathrm{Leb}(\mathcal R)\leq 1-(1-c\varrho\sqrt\varepsilon)^N
\]
for some $c>0$. Hence $\mathrm{Leb}(\mathcal R)\to 1$ as $N\to \infty$.

\subsection{(Partial) verification of the claim}
In this section we provide a partial justification of the claim above. We introduce the notation
\[
\hat\xi_j=(\xi_1,0,\dots,0,\xi_j,0,\dots,0)\qquad \hat J_j=(J_1,0,\dots,0,J_j,0,\dots,0)
\]
and 
\[
\hat \xi=\xi-\hat\xi_j\qquad\qquad \hat J=J-\hat J_j.
\]
Then, we rewrite \eqref{eq:findim} as 
\[
H_\xi=H_{\xi}^{(j)}+\hat\xi\cdot\hat J+\frac 12 \hat J\cdot\hat J+\varepsilon \sum_{i\neq j}m_1m_i\cos (\varphi_1-\varphi_i).
\]
with 
\[
H_\xi^{(j)}(\varphi_1,\varphi_j,J_1,J_j)=\hat \xi_j\cdot \hat J_j+\frac12
\hat J_j\cdot\hat J_j+\varepsilon m_j\cos(\varphi_1-\varphi_j).
\]
The basic observation is that, up to order $O(\varepsilon \delta)$, the dynamics in the variables $1,j$ is governed by the Hamiltonian $H_\xi^{(j)}$. Namely, for any $(\varphi,J)\in\mathbb T_\sigma^N\times B_\rho^N$
\begin{equation}\label{eq:vfield}
\dot\varphi_1=\xi_1+J_1\qquad \dot J_1=O(\varepsilon \delta)\qquad \dot \varphi_j=\xi_j+J_j\qquad\dot J_j=\varepsilon \sin(\varphi_1-\varphi_j)+O(\varepsilon \delta).
\end{equation}
We now introduce the  change of variables 
\[
\theta_1=\varphi_1\qquad \theta_j=\varphi_j-\varphi_1\qquad I_1=J_1+J_j\qquad\qquad I_j=J_j.
\]
In these coordinates the vector field \eqref{eq:vfield} reads
\begin{equation}\label{eq:nhimflow}
\dot\theta_1=\xi_1+I_1-I_j\qquad \dot I_1=O(\varepsilon\delta)  \qquad\dot\theta_j=(\xi_j-\xi_1)+2I_j-I_1 \qquad\dot I_j=-\varepsilon\sin\theta_j+O(\varepsilon\delta)
\end{equation}
For $\delta=0$ the vector field \eqref{eq:nhimflow} 
exhibits a two-dimensional normally hyperbolic cylinder
\[
\gamma_j=\{\theta_j=0,\ I_j=\frac 12(I_1+\xi_1-\xi_j)\colon (\theta_1,I_1)\in\mathbb T\times B\}\subset\mathbb T^2\times \mathbb R^2
\]
with stable and unstable manifolds $W^{s,u}(\gamma_j)$ which intersect at $\gamma_j$ forming an angle
\[
\angle(T_{\gamma_j} W^s(\gamma_j),T_{\gamma_j} W^u(\gamma_j))\gtrsim \sqrt\varepsilon.
\]
Hence, by the theory of normal hyperbolicity (see, for instance, \cite{MR501173}) one can show that there exists $\delta_0(\varepsilon)$ not depending\footnote{It is at this point that our argument is not fully rigorous. 
Indeed, showing that one can chose $\delta_0(\varepsilon)$ not depending on $N$ requires some (tedious but standard) additional work.} on $j,N$ such that for any 
\[
0<\delta\leq \delta_0(\varepsilon)
\]
\begin{itemize}
    \item there exists a normally hyperbolic manifold $\Lambda_j\subset \mathbb T^N\times B^N$ which can be parametrized as 
    \[
    \Lambda_j=\{\theta_j=\alpha_j(\theta_1,I_1,\hat\varphi,\hat J),\ I_j=\beta_j(\theta_1,I_1,\hat\varphi,\hat J)\colon (\theta_1,I_1)\in T\times B,\ (\hat\varphi,\hat J)\in\mathbb T^{N-2}\times B^{N-2}\}.
    \]
    for some functions 
    \[
    \alpha_j=O_{C^0}(\delta)\qquad\qquad \beta_j=\frac 12(I_1+\xi_1-\xi_j)+O_{C^0}(\delta).
    \]
    \item The invariant manifolds $W^{u,s}(\Lambda_j)$ are $O_{C^0}(\delta)$ perturbations of the manifolds $W^{s,u}(\gamma_j)\times(\mathbb T^{N-2}\times B^{N-2})$.
\end{itemize}
Notice that the manifolds $W^{u,s}(\Lambda_j)$ have codimension one. Moreover, for $|I_1|\leq C\sqrt\varepsilon$ and $|\xi_1-\xi_j|\leq \varrho \sqrt\varepsilon$ with $C,\varrho>0$ sufficiently small, these manifolds cross the section $\{J=0\}$. Hence, (possibly after decreasing even more $C$ and $\varrho$) any graph of the form
\[
\mathcal T=\{(\varphi,\Phi(\varphi))\colon \varphi\in\mathbb T^N\}
\]
which is contained in a $C\sqrt\varepsilon$ neighbourhood of the origin must intersect one of the manifolds $W^{u,s}(\Lambda_j)$. In particular, no Lagrangian invariant torus can be contained in a $C\sqrt\varepsilon$ neighbourhood of $\{J=0\}$.

\section{Main novelties and challenges}\label{sec:comparison}

One of the very first applications of KAM techniques to construct quasiperiodic invariant tori in high dimensional lattice systems appeared in the work of Wayne \cite{MR769350}. In this work the author studied short range perturbations of  integrable Hamiltonians, a prototipical example being given by (with $a\cdot b=\sum_i a_i b_i$)
\begin{equation}\label{eq:nearest}
H(\theta,I)=\tfrac{1}{2}\, I\cdot I+\varepsilon\sum_{i=1}^{N-1} \cos(\theta_{i+1}-\theta_i) \qquad\qquad (\theta,I)\in\mathbb T^N\times\mathbb R^N.
\end{equation}
The particular structure of the interaction suggests that for these systems, KAM tori might exist for values of $\varepsilon$ much larger than the threshold $(N!)^{-\alpha}$ (for some $\alpha>0$) implied by the classical estimates. Indeed, the dependence on the dimension enters the iterative KAM scheme through the small divisors, which get worse as the dimension increases. However, for \eqref{eq:nearest}, the small divisors (at least at the first iteration of the KAM scheme) involve only integer combinations in $\mathbb Z^2$ (instead of $\mathbb Z^N$). Exploiting this observation Wayne was able to show that for systems as in \eqref{eq:nearest}, KAM tori exist for $\varepsilon\sim N^{-\alpha}$ for some $\alpha>0$ large enough. 

Still within the framework of short range interactions, an elegant extension of these ideas for infinite dimensional systems  was later given by Pöschel under the framework of systems with \textit{spatial structure} \cite{PoschelSpatialstructure}. 

\black 

 The main novelty in Theorem \ref{thm:mainlagrangianprecise} is that the perturbation $P$ is of long range. Namely, one can see from equations \eqref{eq:vectorfieldnormalform} that the vector field does not decay with $n\to\infty$.  We have seen in Section \ref{sec:intro} (see also Section \ref{sec:particles}) that long range perturbations naturally appear when studying 
systems of interacting particles.

 Theorem \ref{thm:mainlagrangianprecise} can be seen as an extension of some of the ideas in \cite{PoschelSpatialstructure} to a system with long range interactions. Indeed,  \cite{PoschelSpatialstructure} deals with Hamiltonians of the form \eqref{eq:normalformLag} but requires $P$ to satisfy 
 \begin{equation}\label{eq:poschel}
\sum_{A\in \mathbb N_0} (m_A)^{-\beta}|P_A|_{\rho,\sigma}<\infty
 \end{equation}
 for some $\beta>1$ (and some $\rho,\sigma>0)$, {   while in our case $\beta=1$}. For  perturbations as in \eqref{eq:poschel}, one readily sees that 
 \[
\dot\varphi_j=\omega_j+\frac{1}{m_j}\partial_{J_j} P=\omega_j+\mathcal O(m_j^{\beta-1})\qquad\qquad \dot J_j=-\frac{1}{m_j}\partial_{\varphi_j} P=O(m_j^{\beta-1})
 \]
 \color{black}That is, the vector field decays exponentially fast (recall that $m_A=O(\exp(-\kappa\sum_{i\in A} i))$ \black ) as $n\to\infty$.  This, as we said before, fails in our case where $\beta=1$.
 
 \begin{rem}
Another classical related work dealing with lattice systems is \cite{Frohlichlocalization}. The interested reader is invited to check the introduction of that paper for a  survey on the theory of transport in anharomonic crystals and the role of KAM theory to explain the absence of transport.

However, from the functional setting viewpoint, in that work the interactions considered are of nearest neighbours. In particular, the vector field decays exponentially fast (as in \cite{PoschelSpatialstructure}). Hence, the techniques developed in that paper cannot be applied to deduce our main result, i.e. Theorem \ref{thm:mainlagrangianprecise}. 
 \end{rem}

\medskip 
The main difficulties in extending P\"oschel's theory to perturbations $P\in\mathcal Y_{\beta,\rho,\sigma}$ (recall that this functional space is defined in \eqref{eq:defnBanachspaceLag}) are the following:

\begin{itemize}[leftmargin=*]
    \item \textit{Action of the Poisson bracket.}  
    To carry out the KAM iteration in spaces of the form $\mathcal Y_{\beta,\rho,\sigma}$, one needs to show that the Poisson bracket  
    $\displaystyle
    \{h,g\}=\sum_{j\in\mathbb N} m_j^{-1}\big( \partial_{J_j} h\,\partial_{\varphi_j} g-\partial_{\varphi_j} h\,\partial_{J_j} g\big)
    $
    defines a bounded bilinear operator 
    \[
    \{\cdot,\cdot\}:\mathcal Y_{\beta,\rho,\sigma}\times\mathcal{Y}_{\beta,\rho,\sigma}\to \mathcal{Y}_{\beta,\rho',\sigma'}
    \]
    (with $0<\rho'<\rho$ and $0<\sigma'<\sigma$).  
    A similar statement is proved in \cite{PoschelSpatialstructure} for functions belonging to the Banach space induced by the norm \eqref{eq:poschel}.  
    However, since the norm defining $\mathcal Y_{\beta,\rho,\sigma}$ involves a supremum, a non-trivial modification of P\"oschel's argument is required.  

    \item \textit{Diophantine inequalities.}  
    The main reason why we do not expect (recall the discussion in Section \ref{sec:stochlayer})
    that Theorem \ref{thm:mainlagrangianprecise} holds for almost all $\xi\in[a,b]^{\mathbb N}$ (with respect to the product 
    measure) is that, due to the long range nature of the interactions, the KAM scheme requires Diophantine conditions that are \emph{uniform} in $j$ (the largest index appearing in each monomial).  
    For instance, one must impose an estimate of the form, for some $\tau>0$,
    \begin{equation}\label{eq:diophantineheuristic}
    |\xi_1 l+\xi_j k|\;\gtrsim\; |(l,k)|_1^{-\tau},
    \end{equation}
    valid for all $l,k\in\mathbb Z^2\setminus\{0\}$ and all $j\in\mathbb N$.  
    In other words, each small-divisor condition removes, a priori, a uniform amount of measure (there is no decay in $j$ in the lower bound of \eqref{eq:diophantineheuristic}).  
    To overcome this difficulty, we localize around frequencies  $\xi\in[a,b]^{\mathbb N}$ of box dimension strictly smaller than one.  
    In sufficiently small neighborhoods of such frequencies, the measure discarded in order to ensure summability of the conditions \eqref{eq:diophantineheuristic} over $j$ can be controlled (as well as the other Diophantine conditions required to control the small divisors, see Section \ref{sec:measureestimatesLag}).  
\end{itemize}
\vspace{0.2cm}

  Let us draw now a comparison between our results (in particular, the applications in Section \ref{sec:actionanglenormalform}) and the very recent work of Corsi, Gentile and Procesi \cite{RenormalizationGroupMaxtori}, where the authors establish the existence of full-dimensional tori in infinite-dimensional mechanical systems. Under suitable conditions on the interaction law (to be specified below), they employ the so-called \emph{tree formalism} (in place of the classical KAM scheme) to construct full-dimensional tori whose frequencies satisfy a generalized Bryuno condition.  

More precisely, given a sequence $\bs h=\{h_j\}$ with $h_j\to \infty$ and some $q>0$, define
\[
\mathbb T_{\sigma,\bs h}=\Big\{\theta\in (\mathbb C/2\pi\mathbb Z)^\mathbb Z : |\operatorname{Im}\theta_j|\leq \sigma h_j\Big\}, 
\qquad 
\ell_{\infty,q}=\Big\{I=\{I_j\}: I_j\in\mathbb R,\ \sup_j |I_j|(1+|j|)^q<\infty\Big\}.
\]
They study the Hamiltonian (with $a\cdot b=\sum_i a_i b_i$)
\[
H(\theta,I)=\tfrac{1}{2}\, I\cdot I+ \varepsilon \black f(\theta), 
\qquad\qquad (\theta,I)\in \mathbb T_{\sigma,\bs h}\times \ell_{\infty,q},
\]
where $f$ admits a holomorphic and bounded extension to $\mathbb T_{\sigma,\bs h}$.  
Typical choices for $\bs h=\{h_j\}$ include $h_j=\log^s(1+|j|)$ with $s>1$ or $h_j=\langle j\rangle^\alpha$ with $\alpha>0$.  In this framework, the vector field
\begin{align*}
\dot \theta &= I, \\
\dot I &= -\varepsilon \partial_\theta f(\theta)
\end{align*}
exhibits strong decay as $j\to\infty$. Indeed, any monomial in $f$ such that $\partial_{\theta_j} f\neq 0$ has size bounded above by $O\big(\exp(-\sigma |h_j|)\big)$. Notice that, in particular, in the naturally scaled coordinates $\widetilde I\in\ell_\infty$ defined by $I_j=(1+|j|)^{-q}\widetilde I_j$ the ratio (compare to the case of long range perturbations in Definition \ref{defn:longrange}, in particular see \eqref{eq:freqpertratio})
\[
\frac{(\text{peturbation})_j}{(\text{unperturbed frequency})_j}\lesssim \varepsilon (1+|j|)^{2q} \exp(-\sigma|h_j|) 
\]
decays at superpolynomial speed as $j\to \infty$.
Therefore, their setting crucially relies on the rapid decay encoded in the sequence $\bs h$, and their results do not appear to cover, at least \emph{a priori}, the case of mechanical Hamiltonians of the type considered in \eqref{eq:fulldimintrolag0}, \eqref{eq:firstHamnew}, or \eqref{eq:firstHam}. On the other hand, the rapid decay they impose allows them to obtain positive measure sets of KAM tori. 

 Recently \cite{tong2025sharpregularityfulldimensionaltori} used similar norms (extended to the Gevrey and smooth category, but again with a corresponding rapid decay in the coordinates' indices) to prove in all regularities an infinite dimensional version of Arnold's theorem on local linearization of Diophantine toral translations. The choice of the norms imposes a spatial structure on the perturbation that forces the dependence on far away indices to decay very rapidly, which allows them to address the small denominators and get linearization for full measure sets of infinite frequency vectors. \black

\bigskip

Finally, comparing our results with KAM results for PDEs is less straightforward. The main differences are the following:

\begin{itemize}
    \item \textit{Structure of the perturbation.}  
    The perturbation in \eqref{eq:normalformLag}  involves (small) parameters, namely the mass vector $\boldsymbol m\in \ell^\mathbb R_{\infty,\kappa}$, whereas in normal forms around an elliptic fixed point, smallness of the perturbation is typically tied to the distance to the fixed point itself.  

    To illustrate this, consider the one-dimensional quintic NLS  with convolution potential \black
    \begin{equation}\label{eq:nlstext}
    H(z,\bar z)=\sum_{n\in\mathbb Z}(n^2+V_n) z_n\bar z_n
    +\varepsilon \!\!\!\sum_{n_1-n_2+n_3-n_4+n_5-n_6=0}\!\! z_{n_1}\bar z_{n_2}z_{n_3}\bar z_{n_4}z_{n_5}\bar z_{n_6}.
    \end{equation}
    One may consider initial data $\{(z_n,\bar z_n)\}_{n\in\mathbb Z}\in \ell_{\infty,\sqrt{\bs m}}^2$\footnote{We abuse notation and write $\ell_\infty$ also for the space of bounded bi-infinite sequences.} for some $\bs m\in \ell^\mathbb R_{\infty,\kappa}$.  
    The change of variables
    \[
    z_n=\sqrt{m_n}\,u_n, \qquad \bar z_n=\sqrt{m_n}\,\bar u_n,
    \]
    conjugates the flow of \eqref{eq:nlstext} on the symplectic manifold 
    \[
    \big(\ell_{\infty,\sqrt{\bs m}}^2,\; i\sum_n \mathrm{d}z_n\wedge \mathrm{d}\bar z_n\big)
    \]
    to that of a Hamiltonian
    \[
    \widetilde H(u,\bar u)=\Omega(V)\,u\cdot \bar u+\varepsilon P
    \]
    on the symplectic phase space
    \[
    \big(\ell_\infty^2,\; i\sum_n m_n \,\mathrm{d}u_n\wedge \mathrm{d}\bar u_n\big),
    \]
    where $\Omega(V)=\mathrm{diag}(\dots,n^2+V_n,\dots)$ and
    \begin{equation}\label{eq:perturbationnlsscaled}
    P=\!\!\!\sum_{n_1-n_2+n_3-n_4+n_5-n_6=0}\!\!\!
    \sqrt{m_{n_1}m_{n_2}m_{n_3}m_{n_4}m_{n_5}m_{n_6}}\;
    u_{n_1}\bar u_{n_2}u_{n_3}\bar u_{n_4}u_{n_5}\bar u_{n_6}.
    \end{equation}
    As shown in Lemma 1.1 of \cite{Bourgainfulldim}, the restriction $n_1-n_2+n_3-n_4+n_5-n_6=0$ implies that, for the choice $m_n=\exp(-\sqrt{|n|})$, all monomials in \eqref{eq:perturbationnlsscaled} satisfy
    \begin{equation}\label{eq:bourgainlemma}
    \sqrt{m_{n_1}m_{n_2}m_{n_3}m_{n_4}m_{n_5}m_{n_6}}
    \;\leq\; m_{n_1^*}\,\prod_{i\geq 3} m_{n_i^*}^{1/4},
    \end{equation}
    where $\{n_i^*\}$ denotes the decreasing rearrangement of $\{|n_i|\}$.  

    The estimate \eqref{eq:bourgainlemma} does not fit into Pöschel's framework (restricted to short range interactions). In fact, Bourgain notes in \cite{Bourgainfulldim} that an estimate of the form \eqref{eq:bourgainlemma} with the right-hand side replaced by $m_{n_1^*}^{1+\delta}$ would suffice for the techniques of \cite{PoschelSpatialstructure}, since it would imply that the vector field at mode $n$ decays as $m_n^\delta$. However, this stronger estimate does not follow from \eqref{eq:bourgainlemma}.  

    Our own framework also does not accommodate \eqref{eq:bourgainlemma}, as we would require an estimate of the form
    \[
    m_{n_1^*}\,\prod_{i\geq 2} m_{n_i^*}^\beta
    \]
    for some $\beta>0$. Bourgain circumvents this issue in \cite{Bourgainfulldim} by exploiting additional arithmetic features of the model, specifically, the so-called \emph{regularizing effect of the resonant monomials} to establish a KAM theorem using only \eqref{eq:bourgainlemma}.    Recently, Cong \cite{MR4800925} has been able to construct orbits of \eqref{eq:nlstext} which decay like $|z_n|\asymp \exp(-\log^\sigma |n|)$ with $\sigma>2$ as $|n|\to\infty$.\black

    \item \textit{Parameter space.}  
    As will be shown in Section \ref{sec:actionanglenormalform}, amplitude
    frequency modulation for (non-degenerate)  integrable \black systems of interacting particles reduces the Hamiltonian to a normal form with external parameters ranging over a domain of the form $[a,b]^\mathbb N$. By contrast, amplitude-frequency modulation around an elliptic fixed point is far more restrictive, as there is much less room for modulation.  
    In PDE settings it is therefore customary to modify the equation and introduce parameters via a convolution potential (see \cite{Bourgainfulldim}). 
    On the other hand, the fact that in \eqref{eq:nlstext} the unperturbed frequencies grow like $n^2$ is certainly advantageous. For instance, lower bounds on small divisors like those in \eqref{eq:diophantineheuristic} are automatic since, in this  case, $\xi_1\sim 1$ and $\xi_j\sim j^2$.
    
   The construction of finite dimensional invariant tori without external parameters was first achieved by Kuksin and Pöschel in \cite{MR1370761}. Quite recently  Bernier, Grébert and Robert have been able to extend the amplitude-frequency modulation  technique to establish the existence  of infinite-dimensional tori for the NLS in the circle and withouth external parameters \cite{bernier2025infinitedimensionalinvarianttori}.  
\end{itemize}


\section{Functional setting}\label{sec:functionalsetting}

 Together with the choice of suitable lower bounds on the small divisors, one of the most important parts of the KAM scheme is to understand the action of the Poisson bracket operator  in the functional space in which we want to run the iteration, i.e. $\mathcal Y_{\beta,\rho,\sigma,O}$.  This is what we do in  Section \eqref{sec:poissonbracket}.

\subsection{The Poisson bracket operator}\label{sec:poissonbracket}
The following result  is key to run the KAM iteration on a scale of Banach spaces of the form $\mathcal{Y}_{\beta,\rho,\sigma,O}$. It  shows that, for any $\rho'<\rho$ and $\sigma'<\sigma$, the bilinear operator $\{\cdot,\cdot\}:\mathcal Y_{\beta,\rho,\sigma,O}\times\mathcal{Y}_{\beta,\rho,\sigma,O}\to \mathcal{Y}_{\beta,\rho',\sigma',O}$, given by
\begin{equation}\label{eq:defnPoissonbracket}
\{h,g\}=\sum_{j\in\mathbb N} m_j^{-1}\left(\partial_{\theta_j} h\partial_{I_j} g-\partial_{I_j} h\partial_{\theta_j} g
\right)
\end{equation}
 is bounded. 

\begin{lem}\label{lem:bracketlemma}
Let $0<\rho\leq \rho_*$, $0<\sigma\leq \sigma_*$ and take $h\in\mathcal{Y}_{\beta,\rho,\sigma,O}$ and $g\in\mathcal{Y}_{\beta,\rho_*,\sigma_*,O}$. Then, $\{h,g\}\in\mathcal{X}_{\beta,\rho-r,\sigma-s,O}$ with 
\[
\lVert \{h,g\}\rVert_{\beta,\rho-r,\sigma-s,O}\lesssim C(r,s,\rho,\sigma,\rho_*,\sigma_*,g)\lVert h\rVert_{\beta,\rho,\sigma,O} \lVert g\rVert_{\beta,\rho_*,\sigma_*,O}.
\]
for
\[
C(r,s,\rho,\sigma,\rho_*,\sigma_*):=\max\{r^{-1}(\sigma_*-\sigma+s)^{-1},s^{-1}(\rho_*-\rho+r)^{-1}\}
\]
Moreover, if
\begin{equation}\label{eq:compositionnecessary}
C(r,s,\rho,\sigma,\rho_*,\sigma_*) \lVert g\rVert_{\beta,\rho_*,\sigma_*,O}\leq \frac 12,
\end{equation}
for any $t\in[0,1]$ we have that $h\circ\phi^t_g\in\mathcal Y_{\beta,\rho-r,\sigma-s,O}$ and 
\[
\lVert h\circ\phi^t_g \rVert_{\beta,\rho-r,\sigma-s,O}\leq (1+2C\lVert g\rVert_{\beta,\rho_*,\sigma_*,O})\lVert h\rVert_{\beta,\rho,\sigma,O}.
\]
\end{lem}

\begin{proof}
Since
 the set $O\subset \mathbb C^\mathbb N$ does not play any role, during this proof we omit the dependence on it. We write $\{h,g\}=f_1+f_2$ with
\[
f_1=\sum_{\overline B\leq \overline A}\{h_A,g_B\},\qquad f_2=\sum_{\overline A< \overline B}\{h_A,g_B\}.
\]
Now, for any $j\in\mathbb{N}$ we define
\[
f_{1,j}=\sum_{\overline A=j} \bigg\{h_A,\sum_{\overline B\leq \overline A} g_B \bigg\}=\sum_{\overline A=j} \ \sum_{(\alpha,l)\in \mathcal{A}} h^{[l]}_{\alpha} \ \bigg\{I^\alpha e(l\theta),\sum_{\overline B\leq \overline A} g_B\bigg\}
\]
so $f_1=\sum_{j\in\mathbb N} f_{1,j}$. Notice moreover that 
\[
\lVert f_1 \rVert_{\beta,\rho-r,\sigma-s}\leq  \sup_{j\in\mathbb N} \ \lVert f_{1,j}\rVert_{\beta,\rho-r,\sigma-s}.
\]
We now estimate $\lVert f_{1,j}\rVert_{\beta,\rho-r,\sigma-s}$. To that end, note that
\begin{align*}
\bigg\{I^\alpha e(l\theta),\sum_{\overline B\leq \overline A} g_B\bigg\}=& \sum_{q\in\mathrm{supp} l} m_q^{-1} l_q \sum_{n=0}^{\overline A-q}\sum_{\overline B=q+n} I^\alpha e(l\theta) \partial_{I_q} g_B 
-  \sum_{q\in\mathrm{supp} \alpha} m_q^{-1}\alpha_q \sum_{n=0}^{\overline A-q}\sum_{\overline B=q+n} I^{\alpha-\mathbbm{1}_q}\  e(l\theta) \partial_{\theta_q} g_B.
\end{align*}
Hence,
\[
\lVert f_{1,j}\rVert_{\beta,\rho-r,\sigma-s}\leq C+D
\]
where 
\begin{align*} 
C=&\frac{1}{m_j}\ \sum_{\overline A=j} \sum_{(\alpha,l)\in \mathcal{A}} |h^{[l]}_{\alpha}| \sum_{q\in\mathrm{supp} l} m_q^{-1}|l_q| \sum_{n=0}^{\overline A-q}\sum_{\overline B=q+n} (m_{\underline A\cup B})^{-\beta} |I^\alpha  e(l\theta) \ \partial_{I_q} g_B|_{\rho-r,\sigma-s}\\
D=&\frac{1}{m_j}\ \sum_{\overline A=j} \sum_{(\alpha,l)\in \mathcal{A}} |h^{[l]}_{\alpha}| \sum_{q\in\mathrm{supp} \alpha} m_q^{-1}\alpha_q \sum_{n=0}^{\overline A-q}\sum_{\overline B=q+n} (m_{\underline A\cup B})^{-\beta} |I^{\alpha -\mathbbm{1}_q}\ e(l\theta) \partial_{\theta_q} g_B|_{\rho-r,\sigma-s}
\end{align*}
 and all summations in the proof of the first part of the lemma are restricted to 
$\overline B\leq \overline A$. Next,
\begin{align*}
|I^\alpha\  e(l\theta) \partial_{I_q} g_B|_{\rho-r,\sigma-s}\leq& |I^\alpha\  e(l\theta)|_{\rho-r,\sigma-s} | \partial_{I_q} g_B|_{\rho-r,\sigma-s}\\
=&\rho^{|\alpha|_1} e^{|l|_1(\sigma-s)} \ \frac{1}{\rho_*-(\rho-r)}|g_B|_{\rho_*,\sigma_*},
\end{align*}
and, by a similar argument,
\[
|I^{\alpha -\mathbbm{1}_r} \ e(l\theta) \partial_{\theta_r} g_B|_{\rho-r,\sigma-s}\lesssim \frac{1}{\sigma_*-(\sigma-s)} (\rho-r)^{|\alpha|_1} e^{|l|_1\sigma}\ |g_B|_{\rho_*,\sigma_*}.
\]
On the other hand, for $\overline B\leq \overline A$ satisfying $\{q\}\in A\cap B$
\[
m_{\underline {A\cup B}}= m_{\underline A} m_{\underline B}\ \frac{m_{\overline B}}{m_{A\cap B}}
\gtrsim
m_{\underline A} m_{\underline B}\ \frac{m_{\overline B}}{m_{q}}.
\]
Therefore,
\begin{align*}
  (\rho_*-\rho+r)  C\leq &\frac{1}{m_j}\ \sum_{\overline A=j} (m_{\underline A})^{-\beta} \sum_{(\alpha,l)\in \mathcal{A}} |h^{[l]}_{\alpha}| \rho^{|\alpha|_1} e^{|l|_1\sigma} \sum_{q\in\mathrm{supp} l} |l_q| e^{-|l|_1s}\\
    &\times\sum_{n=0}^{\overline A-q} \left(\frac{m_{q+n}}{m_q}\right)^{1-\beta}\ \frac{1}{m_{q+n}}\sum_{\overline B=q+n} (m_{\underline B})^{-\beta} |g_B|_{\rho_*,\sigma_*}\\
    \leq & \lVert g\rVert_{\beta,\rho_*,\sigma_*}\ \frac{1}{m_j}\ \sum_{\overline A=j} (m_{\underline A})^{-\beta} \sum_{(\alpha,l)\in \mathcal{A}} |h^{[l]}_{\alpha}| \rho^{|\alpha|_1} e^{|l|_1\sigma} \sum_{q\in\mathrm{supp} l} |l_q| e^{-|l|_1 s}   \sum_{n=0}^{\overline A-q} \left(\frac{m_{q+n}}{m_q}\right)^{1-\beta}\\
    \lesssim & \lVert g\rVert_{\beta,\rho_*,\sigma_*}\ \frac{1}{m_j}\ \sum_{\overline A=j} (m_{\underline A})^{-\beta} \sum_{(\alpha,l)\in \mathcal{A}} |h^{[l]}_{\alpha}|\rho^{|\alpha|_1} e^{|l|_1\sigma} \sum_{q\in\mathrm{supp} l} |l_q| e^{-|l|_1s}\\
    \lesssim & \lVert g\rVert_{\beta,\rho_*,\sigma_*}\ \frac{1}{m_j}\ \sum_{\overline A=j} (m_{\underline A})^{-\beta} \sum_{(\alpha,l)\in \mathcal{A}} |h^{[l]}_{\alpha}|\rho^{|\alpha|_1} e^{|l|_1\sigma} \sup_{l\in\mathbb{Z}^{|A|}} |l|_1 e^{-|l|_1s}\\
    \lesssim & \frac{1}{es}\ \lVert g\rVert_{\beta,\rho_*,\sigma_*}\ \frac{1}{m_j}\ \sum_{\overline A=j} (m_{\underline A})^{-\beta} \sum_{(\alpha,l)\in \mathcal{A}} |h^{[l]}_{\alpha}| \rho^{|\alpha|_1} e^{|l|_1\sigma} \lesssim \frac{1}{es}\ \lVert g\rVert_{\beta,\rho_*,\sigma_*} \lVert h\rVert_{\beta,\rho,\sigma}.
\end{align*}
A similar argument shows that 
\[
D\lesssim \max\{r^{-1}(\sigma_*-\sigma+s)^{-1},s^{-1}(\rho_*-\rho+r)^{-1}\}\ \lVert g\rVert_{\beta,\rho_*,\sigma_*} \lVert h\rVert_{   \beta, \black \rho,\sigma},
\]
and we conclude that 
\[
\lVert f_1 \rVert_{\beta,\rho-r,\sigma-s}\lesssim \max\{r^{-1}(\sigma_*-\sigma+s)^{-1},s^{-1}(\rho_*-\rho+r)^{-1}\}  \lVert g\rVert_{\beta,\rho_*,\sigma_*}\lVert h\rVert_{\beta,\rho,\sigma}.
\]
The first part of the lemma follows since a similar argument shows that 
\[
\lVert f_2 \rVert_{\beta,\rho-r,\sigma-s}\lesssim \max\{r^{-1}(\sigma_*-\sigma+s)^{-1},s^{-1}(\rho_*-\rho+r)^{-1}\} \lVert g\rVert_{\beta,\rho_*,\sigma_*}\lVert h\rVert_{\beta,\rho,\sigma}.
\]
We now establish the second part of the lemma. We write 
\[
h\circ\phi_g=\sum_{n=0}^\infty \frac{1}{n!}\mathrm{ad}^n_{g} h\qquad\qquad\text{where}\qquad \mathrm{ad}^n_{g} h=\{\mathrm{ad}^{n-1}_{g} h,g\}
\]
Applying the first part of the lemma recursively to relate
$\lVert \mathrm{ad}^j_{g} h\rVert_{\beta,\rho-(jr)/n,\sigma-(js/n)}$ to\\
$\lVert \mathrm{ad}^{j-1}_{g} h\rVert_{\beta,\rho-((j-1)r)/n,\sigma-((j-1)s/n)}$ we obtain
\[
\lVert \mathrm{ad}^n_{g} h\rVert_{\beta,\rho-r,\sigma-s}\lesssim \left(n\ \max\{r^{-1}(\sigma_*-\sigma+s)^{-1},s^{-1}(\rho_*-\rho+r)^{-1}\} \lVert g\rVert_{\beta,\rho_*,\sigma_*}\right)^n \lVert h \rVert_{\beta,\rho,\sigma}.
\]
So 
\begin{align*}
\lVert  h\circ \phi_g-h\rVert_{\beta,\rho-r,\sigma-s}\lesssim &\lVert  h\rVert_{\beta,\rho,\sigma}  \sum_{n=1}^\infty 
\frac{1}{n!}\left(n\ \max\{r^{-1}(\sigma_*-\sigma+s)^{-1},s^{-1}(\rho_*-\rho+r)^{-1}\} \lVert g\rVert_{\beta,\rho_*,\sigma_*}\right)^n\\
\lesssim &\lVert  h\rVert_{\beta,\rho,\sigma}  \sum_{n=1}^\infty 
\left(e^{-1} \max\{r^{-1}(\sigma_*-\sigma+s)^{-1},s^{-1}(\rho_*-\rho+r)^{-1}\} \lVert g\rVert_{\beta,\rho_*,\sigma_*}\right)^n
\end{align*}
\hskip38mm
$\displaystyle \lesssim \lVert C(r,s,\rho,\sigma,\rho_*,\sigma_*) h\rVert_{\beta,\rho,\sigma}\lVert g\rVert_{\beta,\rho_*,\sigma_*}$
\end{proof}

 Lemma \ref{lem:bracketlemma} can be seen as an extension of Lemmas C.1 and C.2 in \cite{PoschelSpatialstructure} to the limiting case given by the functional setting \eqref{eq:defnBanachspace}. The fact that \eqref{eq:2normLag} involves a supremum makes the proof slightly more involved compared to the case in  \cite{PoschelSpatialstructure}.

\subsection{Frequency transformations}

As it is well known to experts in perturbation theory, during the KAM iteration one needs (in general) to adjust the frequency vector at each step of the scheme. To that end, given an open subset $O\subset\ell_\infty$, we define the Banach space of real-analytic mappings 
\begin{equation}\label{eq:frequencydomains}
\mathtt B_{O}:=\mathcal H(O,\ell_\infty)=\{\nu:O\to\ell_\infty\colon \nu\text{ is real-analytic and }|\nu|_O<\infty\}
\end{equation}
where
\[
|\nu|_O= \sup_{\omega\in \overline O} |\nu(\omega)|_\infty.
\]
We also introduce the following Lipschitz semi-norm
\[
|\nu|_{Lip,O}=\sup_{\omega,\omega_*\in O} \frac{|(\nu)(\omega)-(\nu)(\omega*)|_\infty}{|\omega-\omega_*|_\infty}
\]
The following trivial observation  will be used repeatedly.
\begin{lem}\label{lem:trivialestimatefreq}
Let $h\in \mathcal X_{\beta,\rho,\sigma,O}$. Then $\nu(h)=(\nu_1(h),\dots,\nu_j(h),....)$ given by 
\[
\nu_j(h)=h^{[0]}_{\mathbbm 1_j,0,0}\qquad\qquad j\in\mathbb N
\]
satisfies
$\displaystyle
|\nu(h)|_O\leq \frac{1}{\rho} \lVert h\rVert_{\beta,\rho,\sigma,O}.
$
\end{lem}


\section{KAM for full-dimensional tori}\label{sec:Kamlagrangian}
\subsection{Inductive step.}

In this section we prove Theorem \ref{thm:mainlagrangianprecise}. The proof follows the standard iteration used in other KAM theorems for Hamiltonian systems (see \cite{PoschelSpatialstructure, Poschelclassical} for instance) but adapted to the functional space introduced in \eqref{eq:defnBanachspace}.

 Due to the long range nature of the problem the fact that there exist frequencies satisfying the desired small divisor inequalities is not trivial and it is proved in Section \ref{sec:measureestimatesLag}.

 \begin{rem}
     We have not tried to optimize the dependence of $\varepsilon$ on the other constants involved in the definition of the Hamiltonians \eqref{eq:normalformLag} 
 \end{rem}

\subsection*{The inductive constants}
Consider positive constants $\varepsilon,\beta, \rho$ and $\sigma$ satisfying
\begin{equation}\label{eq:sigmarhosmallnessLag}
    \varepsilon< \rho^6\sigma^6.
\end{equation}
We let 
\[
\varepsilon_0=\varepsilon,\qquad\qquad \beta_0=\beta,\qquad\qquad \rho_0=\varepsilon^{1/2}\rho\qquad\qquad \sigma_0=\sigma
\]
and for $n\geq 1$ we define the inductive constants
\begin{equation}\label{eq:defninductiveLag}
\begin{split}
\varepsilon_{n+1}= \varepsilon_n^{5/4}\ , \qquad\qquad\beta_{n+1}=\beta_n-\mu_n,\qquad&\qquad \rho_{n}=\varepsilon^{1/2}_n \rho,\qquad\qquad
\sigma_{n+1}=\sigma_{n}-2s_n,
\end{split}
\end{equation}
where
\begin{equation}\label{eq:defninductive2}
\mu_n= \frac{\beta_0}{8n^2},\qquad\qquad s_n=\frac{\sigma_0}{8n^2}.
\end{equation}
We also define the subset 

\begin{equation}\label{eq:An}
 \mathbb{A}_n=\left\{A\in\mathbb N_0\colon (m_A)^{\mu_n}\geq  \frac{\varepsilon_{n+1}}{\varepsilon_n}\right\}
\end{equation}
and the constant
\begin{equation}\label{eq:defninductive3}
 L_{n}=\frac{1}{\sigma_n}\left|\log \left( \frac{\varepsilon_{n+1}}{\varepsilon_n}\right)\right|.
\end{equation}

\begin{rem}
  Throughout this section, given two two functions $f,g:\mathbb N\to \mathbb R_+$ we will write $f\lesssim g$ provided that there exists a constant $C$  such that for all $n\in \mathbb N$ we have $f(n)\leq C g(n)$. 
\end{rem}

\subsection*{Diophantine inequalities}
We now introduce  bounds from below on the small divisors that we will require during the iterative scheme. In 
 Section \ref{sec:measureestimatesLag}, we will exhibit a set of frequencies for which these lower bounds hold.

\begin{defn}\label{defn:diophantinesetLag}
  We say that $\xi\in[a,b]^{\mathbb N}$ belongs to  $\mathcal O_n$  if for any $A\in \mathbb N_0$ with  $\underline{A}\in \mathbb{A}_{n}$ the inequality
    \begin{equation}\label{eq:diophLag}
    |\langle\xi_A,l\rangle|\geq \varepsilon_{n}^{1/12}
    \end{equation}
holds for any $l\in\mathbb{Z}^{|A|}$ such that $0< |l|_1\leq L_{n}$. 
\end{defn}
\begin{rem}
  In the PDE literature, a popular choice for the Diophantine inequalities is to require that 
  \[
  |\langle\xi, l\rangle|\geq \kappa \prod_{j\in\mathbb Z} (1+l^2_j j^4)^{-1}
  \]
  for any $l\in\mathbb Z^{\mathbb N}\setminus\{0\}$ with finite support. We notice that this choice is not suitable for long range interactions 
because the vector field in this case does not decay in $j$.
  On the contrary, our definition of the Diophantine inequalities \eqref{eq:diophLag} is 
uniform in the maximum index of $l$ so we can deal with long range perturbations.\end{rem}
In Section \ref{sec:measureestimatesLag}
we show that the set $\mathcal O_F=\bigcap_{n\in\mathbb N} \mathcal O_n$ is not empty. In order to run the KAM scheme it will be convenient to also define the family of closed sets 
\begin{equation}\label{eq:On}
 O_n=\{\xi\in \ell_\infty^N\colon \exists \xi_*\in \mathcal O_n\text{ such that }|\xi-\xi_*|_\infty\leq h_n\}
\end{equation}
where 
\[
h_n= \frac {\varepsilon_n^{1/12}}{2 L_n}.
\]
The following lemma is straightforward. 
\begin{lem}\label{lem:complexfrequenciesLag}
For any $\xi\in O_n$ and any  $A\in \mathbb{A}_{n}$ the inequality
    \[
    |\langle\xi_A,l\rangle|\geq \frac 12 \varepsilon_{n}^{1/12}
    \]
    holds for any $l\in\mathbb{Z}^{|A|}$   such that $0<|l|_1\leq L_{n}$. 
\end{lem}

 We are now ready to proceed with the iterative scheme. We shall use the following notation.
\begin{itemize}
    \item We write $\mathcal{Y}_n$ for $\mathcal{Y}_{\beta_n,\rho_n,\sigma_n,O_n}$.
    \item We write $\mathcal{Y}_{n+1/2}$ for $\mathcal{Y}_{\beta_{n+1},2\rho_{n+1} ,\sigma_n-s_n,O_n}$. Trivially  $\mathcal Y_{n+1}\subset\mathcal Y_{n+1/2}$.
\item We write $\mathtt B_n$ for $\mathtt B_{O_n}$  where $\mathtt B_{*}$ is defined by \eqref{eq:frequencydomains}.
\end{itemize}

\subsection*{The inductive step}

Let now $n\geq 0$ and suppose that for each $\xi\in O_n$, there exists  a Hamiltonian of the form
\[
H^{(n)}_\xi=N_\xi+ P^{(n)},
\]
with  
\[
N_\xi=\xi\cdot J
\]
and $P^{(n)} \in\mathcal{Y}_{n}$ with $\lVert P^{(n)}\rVert_n{\color{red} \leq} \varepsilon_n$. 
\vspace{0.3cm}

\begin{prop}\label{prop:maininductionLag}
Under the hypothesis above, there exist:
\begin{itemize}
\item A map $\varphi_{n+1}:O_{n+1}\to{O}_n$ of the form $\varphi_{n+1}=\mathrm{id}+v_{n+1}$ with $v_{n+1}\in  \mathtt B_{O_{n+1}}$ and $|v_{n+1}|_{O_{n+1}} \lesssim \varepsilon_n^{1/2}$. 

\item A function   $G_{n+1}\in\mathcal{Y}_{n}$  affine in $J$\black and  such that 
\[
H^{(n+1)}_\xi:=H^{(n)}_{\varphi_n(\xi)}\circ\phi_{G_{n+1}}=N_\xi+ P^{(n+1)}
\]
with 
$\displaystyle
\lVert P^{(n+1)}\rVert_{n+1}\leq \varepsilon_{n+1}.
$
\end{itemize}

\end{prop}

We now prove Proposition \ref{prop:maininductionLag}. Then, in Section \ref{sec:endproofLag}, we complete the proof of Theorem \ref{thm:mainlagrangianprecise}. The proof of the proposition, which constitutes the main bulk of the induction, is divided in several steps.
\begin{itemize}
    \item Truncation of the Hamiltonian $H^{(n)}_\xi$.
    \item Solution to the linearized and truncated cohomological equation.
    \item Estimation of the new error term.
    \item Correction of the frequency vector.
\end{itemize}

\subsection*{Truncation of the Hamiltonian $H^{(n)}_\omega$}
We define
\[
Q_n= \sum_{\underline A\in \mathbb{A}_n} \sum_{\substack{(l,\alpha)\in\mathcal{A}\\ \ |l|_1\leq L_n\\  |\alpha|_1\leq 1}} (P^{(n)})_{\alpha}^{[l]}\  J^\alpha   e(l\varphi).
\]
The truncation is chosen so that the map $J\mapsto Q_n(\varphi,J)$ is affine and the following estimates hold.
\begin{lem}\label{lem:truncationlemmaLag} 
The term $R_n=P^{(n)}-Q_n$ belongs to $\mathcal{Y}_{n+1/2}$ and satisfies
\[
\lVert R_n\rVert_{n+1/2}\lesssim \varepsilon_{n+1}.
\]
\end{lem}
\begin{proof}
We have 
\begin{align*}
    \lVert R_n \rVert_{n+1/2 \black} 
    \lesssim & \left( \min_{\underline A\notin\mathbb{A}_n} (m_A)^{\mu_n}+ \left(\frac{\rho_{n+1}}{\rho_n}\right)^2+ \exp (-L_n \sigma_n) \right) \lVert P^{(n)}\rVert_n  
 \lesssim  \left(  \frac{\varepsilon_{n+1}}{\varepsilon_n}+ \frac{\varepsilon_{n+1}}{\varepsilon_n}+ \frac{\varepsilon_{n+1}}{\varepsilon_n}\right)\varepsilon_n\lesssim \varepsilon_{n+1}.
    \end{align*}
Indeed, the three terms between parenthesis correspond respectively to:
\begin{itemize}
    \item the contribution of the monomials supported on subsets $A$ for which $\underline A\notin \mathbb A_n$,
    \item the contribution of monomials supported on subsets $A$ for which $\underline A\in \mathbb A_n$ but for which $|\alpha|_1\geq 2$,
    \item the contribution of monomials supported on subsets $A$ for which $\underline A\in \mathbb A_n$ and $|\alpha|_1\leq 1$ but for which $|l|_1>L_n$.
\end{itemize}
The definition of $\rho_n$ in \eqref{eq:defninductiveLag}, $\mathbb A_n$ in \eqref{eq:An} and $L_n$ in \eqref{eq:defninductive3} imply the desired estimates.
\end{proof}

\subsection*{The cohomological equation}
Given $G_{n+1}\in\mathcal{Y}_{n}$, if $\lVert G_{n+1}\rVert_{n}$ is sufficiently small, we can write
    \begin{align*}
    H^{(n)}_\xi\circ\phi_{G_{n+1}}=&N_\xi+\{N_\xi,G_{n+1}\}+Q_n +\int_{0}^{1}\{(1-t)\{N_\xi,G_{n+1}\}+Q_n,G_{n+1}\}\circ\phi^t_{G_{n+1}}\mathrm{d}t+ R_n\circ\phi_{G_{n+1}}
    \end{align*}
 Thus, if we can find $G_{n+1}$ such that $\{N_\xi,G_{n+1}\}+Q_n=\langle Q_n\rangle$, 
where
\begin{equation}\label{eq:averageLag}
\langle Q_n\rangle=\sum_{j\in\mathbb N}  (P_n)_{\mathbbm 1_j}^{[0]}\     J_j.\black
\end{equation}
we obtain that
    \[
    H^{(n)}_\xi\circ\phi_{G_{n+1}}=N_\xi+\langle Q_n\rangle+\int_{0}^{1}\{(1-t)\langle Q_n\rangle+tQ_n,G_{n+1}\}\circ\phi^t_{G_{n+1}}\mathrm{d}t+ R_n\circ\phi_{G_{n+1}}
   \]
    which is again in normal form $H^{(n)}_\xi\circ\phi_{G_{n+1}}=N_{n+1,\xi} +P_{n+1}$ with 
    \begin{equation}\label{eq:newerrortermLag}
     N_{n+1,\xi} =N_\xi +\langle Q_n\rangle,\qquad\qquad
   P_{n+1}= \int_{0}^{1}\{(1-t)\langle Q_n\rangle+tQ_n,G_{n+1}\}\circ\phi^t_{G_{n+1}}\mathrm{d}t+R_n\circ\phi_{G_{n+1}}\\
    \end{equation}

We now study the existence and properties satisfied by the solution to the linearized cohomological equation.
\begin{lem}\label{lem:solutioncohomolequLag}
There exists $G_{n+1}\in\mathcal{Y}_{n}$  affine in $J$ \black such that 
\begin{equation}\label{eq:linearizedcohomeqLag}
\{N_\xi,G_{n+1}\}+ Q_n=\langle Q_n\rangle
\end{equation}
Moreover, 
\[
\lVert G_{n+1}\rVert_{n}\lesssim \varepsilon_n^{11/12}.
\]
\end{lem}

\begin{proof}
The function $G_{n+1}$ defined by 
\[
G_{n+1}= i\sum_{\underline A\in\mathbb{A}_n} \sum_{\substack{(l,\alpha)\in\mathcal{A}\\ 
0<|l|_1\leq L_n\\  |\alpha|_1\leq 1}} (\langle \xi_A, l\rangle)^{-1}\  P_{\alpha}^{[l]}\  J^\alpha  e(l\varphi)
\]
is clearly a solution to \eqref{eq:linearizedcohomeqLag}. Next, Lemma \ref{lem:complexfrequenciesLag} guarantees that the coresponding small divisors are lower bounded by $\varepsilon_n^{-1/12}$. Hence
$\displaystyle \lVert G_{n+1}\rVert_{n} \leq  \varepsilon_n^{-1/12}\lVert P_n \rVert_{n}\lesssim  \varepsilon_n^{11/12}.$
\end{proof}

\subsection*{Analysis of the new error term}
We now provide an estimate for the new error term $P_{n+1}$.

    \begin{lem}\label{prop:newerrorregularLag}
        Let $P_{n+1}$ be the function defined in \eqref{eq:newerrortermLag}. Then, $P_{n+1}\in\mathcal{Y}_{n+1}$ and satisfies
        \[
        \lVert P_{n+1} \rVert_{n+1}\lesssim \varepsilon_{n+1}.
        \]
    \end{lem}

\begin{proof}
Recall that
\[
\lVert Q_n \rVert_{n}\lesssim    \varepsilon_n,\qquad\qquad \lVert G_{n+1} \rVert_{n}\lesssim \varepsilon_n^{11/12},\qquad\qquad \lVert R_n \rVert_{n+1/2}\lesssim \varepsilon_{n+1}.
\]
Hence, it follows from the definition of the inductive constants $\rho_n$ and $\sigma_n$ that (recall also the definition of $\mathcal Y_{n+1/2}$ and compare the expression below to the condition \eqref{eq:compositionnecessary})
\begin{align*}
(2\rho_{n+1}-\rho_{n+1})^{-1}s_n^{-1} \lVert  G_{n+1}\rVert_{n}\lesssim & \rho_{n+1}^{-1} \sigma^{-1} n^2\lVert  G_{n+1}\rVert_{n}\lesssim \sigma^{-1} n^2\rho_{n+1}^{-1}\varepsilon_n^{11/12}\\
\lesssim & n^2 \sigma^{-1} \rho^{-1}\varepsilon_n^{11/12-5/8}\ll 1,
\end{align*}
where we have made use of the hypothesis \eqref{eq:sigmarhosmallnessLag}. Thus, 
the second item in Lemma \ref{lem:bracketlemma} shows that 
\[
\lVert R_n\circ\phi_{G_{n+1}}\rVert_{n+1}\lesssim \lVert R_n\rVert_{n+1/2}\lesssim \varepsilon_{n+1}.
\]
On the other hand, the first item in Lemma \ref{lem:bracketlemma} shows that, since $\varepsilon_n<\varepsilon_0$, in view of the hypothesis \eqref{eq:sigmarhosmallnessLag}
\begin{align*}
\lVert \{(1-t)\langle Q_n\rangle+tQ_n,G_{n+1}\}\rVert_{\beta_{n+1},\rho_n/2,\sigma_n-s_n}&\lesssim \rho_n^{-1} s_n^{-1} \lVert Q_n\rVert_{n}\lVert G_{n+1}\rVert_{n}\\
&\lesssim  \rho_n^{-1} \sigma^{-1}\varepsilon_n^{2-1/12}
= \sigma^{-1}\rho^{-1}\varepsilon_{n}^{17/12}
\lesssim  \varepsilon_{n+1},
\end{align*}
so, for any $t\in[0,1]$, it follows from the second item in Lemma \ref{lem:bracketlemma} that\\
$\displaystyle \lVert \{(1-t)\langle Q_n\rangle+tQ_n,G_{n+1}\}\circ\phi^t_{G_{n+1}}\rVert_{n+1}\lesssim \varepsilon_{n+1}.$
\end{proof}

\subsection*{Correction of the frequency vector}
Due to (6.8) and the definition of the mass scalar product in \eqref{def.massproduct}, the ``new'' normal form $N_{n+1,\xi}$ in \eqref{eq:newerrortermLag} can be written as
\[
N_{n+1,\xi}(J)=(\xi+\tilde v_{n+1}(\xi))\cdot J
\]
where $\tilde v_{n+1}\in\ell_\infty^{N}$ is given by
\[
(\tilde v_{n+1})_j= m_j^{-1}(P_n)^{[0]}_{\mathbbm 1_j}.
\]

\subsubsection*{Transforming the tangential frequencies}

\begin{lem}\label{lem:transffrequbddLag}
 There exists $ v_{n+1}\in \mathtt B_{O_{n+1}}$ such that $\varphi_{n+1}=\mathrm{id}+v_{n+1}$ satisfies $(\mathrm{id}+\tilde v_{n+1})\circ\varphi_{n+1}=\mathrm{id}$ on $O_{n+1}$   and $\varphi_{n+1}(O_{n+1})\subset O_n$\black. Moreover,
\[
|v_{n+1}|_{O_{n+1}}\lesssim \varepsilon_n^{1/2}\qquad \qquad| v_{n+1}|_{Lip,O_{n+1}}\lesssim \varepsilon_n^{1/2}/h_n.
\]
\end{lem}

\begin{proof}
Let $\widehat O_n$ be defined as in \eqref{eq:On} but with $h_n$ substituted by $h_n/2$. Notice that, in view of Lemma \ref{lem:trivialestimatefreq} we have
\[
|\tilde v_{n+1}|_{\widehat O_{n}}\lesssim \varepsilon_n/\rho_n\lesssim \varepsilon_n^{1/2} \ll h_n.
\]
Then, the proof boils down to a straightforward application of Lemma \ref{lem:inverse}.
\end{proof}

We finally compose with the map $\varphi_{n+1}$ in parameter space and define
\[
N_\xi=N_{n+1,\varphi_{n+1}(\xi)}\qquad  \qquad P^{(n+1)}(\cdot;\xi)=P_{n+1}(\cdot;\varphi_{n+1} (\xi))
\]
to close the inductive step.

\subsection{Initialization and convergence of the inductive scheme}\label{sec:endproofLag}
\subsection*{Initialization} 
Let $\varepsilon,\beta\rho,\sigma$ be positive constants satisfying \eqref{eq:sigmarhosmallnessLag} and consider a Hamiltonian $H_\xi$ as in \eqref{eq:normalformLag}. The argument used in Lemma \ref{lem:truncationlemmaLag} shows that $P=P_h+\varepsilon\widetilde P$ satisfies
\[
\lVert P_h\rVert_{0}\leq \left(\frac{\rho_0}{\rho}\right)^{2}\lVert P_h\rVert_{\beta,\rho,\sigma,O}\leq \left(\frac{\rho_0}{\rho}\right)^{2}=\varepsilon_0
\]
so
\[
\lVert P\rVert_{0}\lesssim \varepsilon_0.
\]
Hence, the Hamiltonian $H^{(0)}_\xi=H_\xi$ satisfies all the hypothesis in   Proposition \ref{prop:maininductionLag}\black. This proposition can now be applied in an inductive manner. Indeed, notice that  since $0<\varepsilon_0<1$ the sequence  $\varepsilon_n\to 0$ superexponentially.

\subsection*{Convergence of the limiting transformation}
The convergence of the infinite product of conjugacies follows from standard arguments and we dispose of it quickly (more details about the ideas outlined in this section can be found in \cite{Poschelclassical}). We start by noticing that due to their particular structure,  the conjugacies in Proposition \ref{prop:maininductionLag} can be extended to a larger domain. The proof of the following result is straightforward.
\begin{lem}\label{lem:conjugacyextensionlag}
    Let $\beta,\rho,\sigma$ be positive constants satisfying \eqref{eq:sigmarhosmallnessLag} and let $H_\xi\in \mathcal Y_{\beta,\rho,\sigma,[a,b]^{\mathbb N}}$ be as in \eqref{eq:normalformLag}. For any $n\in\mathbb N$ let $G_{n+1}$ be the conjugacy obtained  after $n$ applications of Proposition \ref{prop:maininductionLag}. Then, for any $n\in\mathbb N$  and  any $\tilde\rho\geq \rho_n$we have $G_{n+1}\in\mathcal Y_{\beta_n, \tilde\rho,\sigma_n,O_n}$  and
    \begin{equation}\label{eq:extendedconjugacyestimateLag}
        \lVert G_{n+1}\rVert_{\beta_n,\tilde\rho,\sigma_n,O_n}\leq \frac{\tilde\rho}{\rho_n}\lVert G_{n+1}\rVert_n \lesssim  \frac{\tilde\rho}{\rho}\varepsilon_n^{1/3}.
    \end{equation}
\end{lem}

In particular, the inductive sequence of Hamiltonians $H^{(n)}_\omega$ obtained by iteration of   Proposition \ref{prop:maininductionLag} \black can be extended to a larger domain.  We now study the limit Hamiltonian obtained by iterating the construction  Proposition \ref{prop:maininductionLag} infinitely many times.   To do so, instead of considering $\rho_n\to 0$ as we did in the inductive step, we will introduce a different sequence $\tilde\rho_n\to\rho/2$ and decompose $P^{(n)}$ as the sum of a term $Q^{(n)}$, affine in $J$, and the quadratic remainder $R^{(n)}$. In order to complete the proof of Theorem \ref{thm:mainlagrangianprecise} we will show that, in the complex domains of size $\tilde\rho_n$,   $Q^{(n)}\to 0$  and $G_n\to 0$ sufficiently fast (for instance, $\lVert G_n\rVert_{\beta_n,\rho_n,\sigma_n,O_n}\leq 2^{-n}$)  while the quadratic part $R^{(n)}$ remains bounded. \black

\begin{prop}\label{prop:secondinduction}
Let $\beta,\rho,\sigma$ be positive constants satisfying \eqref{eq:sigmarhosmallnessLag} and let $H_\xi\in \mathcal Y_{\beta,\rho,\sigma,[a,b]^{\mathbb N}}$ be as in \eqref{eq:normalformLag}.  For any $n\in\mathbb N$ let $\varphi_{n+1}$ and $G_{n+1}$ be as in Proposition \ref{prop:maininductionLag}. Let 
$\displaystyle \mathcal O_F=\bigcap_n\mathcal O_n$ with $\mathcal O_n$ as in Definition~\ref{defn:diophantinesetLag}. Then, for any $\xi\in \mathcal O_F$, the sequence $\{H^{(n)}_\xi\}_{n\in\mathbb N}$ defined inductively by \black
\[
H^{(n+1)}_\xi= H^{(n)}_{\varphi_{n+1}(\xi)}\circ \phi_{G_{n+1}}\qquad\qquad H^{(0)}_\xi=H_\xi
\]
    converges  to a  Hamiltonian $H_\xi\in \mathcal Y_{\beta/2,\rho/2,\sigma/2}$ of the form \eqref{eq:normalformLag}.
\end{prop}

\begin{proof}
  By Proposition \ref{prop:maininductionLag}, for all $n$ we have  $H_\xi^{(n)}\in \mathcal Y_{\beta_n,\rho_n,\sigma_n,O_n}$ and $H_\xi^{(n)}=N_\xi+P^{(n)}$ with 
\begin{equation}\label{eq:auxiliaryconvergenceLag}
\lVert P^{(n)}\rVert_{n}\lesssim  \varepsilon_n.
\end{equation}
    Let $\tilde\rho_n=\rho(1-\sum_{j=1}^n 1/2^{j})$. We point out that, in particular,  the estimate \eqref{eq:auxiliaryconvergenceLag} implies that for all $n\in\mathbb N$
\begin{equation}\label{eq:qnLag}
Q^{(n)}= \sum_{\substack{(l,\alpha)\in\mathcal{A}\\   |\alpha|_1\leq 1}} (P^{(n)})_{\alpha}^{[l]}\  J^\alpha  e(l\varphi)
\end{equation}
satisfies
\begin{equation}\label{eq:qnestimateLag}
\lVert Q^{(n)}\rVert_{\beta_n,\tilde\rho_n,\sigma_n,O_n}\leq  \frac{\tilde\rho_n}{\rho_n}\lVert P^{(n)}\rVert_{\beta_n, \rho_n,\sigma_n,O_n} \lesssim  \frac{\rho}{\rho_n}\varepsilon_n= \rho \varepsilon_n^{1/2}.
\end{equation}
Let 
\begin{equation}\label{eq:reminderLag}
R^{(n)}=P^{(n)}-Q^{(n)}
\end{equation}
and suppose now  that for some $n\in\mathbb N\cup\{0\}$ we have  $R^{(n)}\in \mathcal Y_{\beta_n,\tilde\rho_n,\sigma_n,O_n}$ and
\[
\lVert R^{(n)}\rVert_{\beta_n,\rho,\sigma_n,O_n}\leq (1+\sum_{j=1}^n \varepsilon_j^{1/8}) .
\]
The estimate  \eqref{eq:extendedconjugacyestimateLag} implies that (compare to the condition \eqref{eq:compositionnecessary} in Lemma \ref{lem:bracketlemma})
\begin{equation}\label{eq:smallchangeLag}
\frac{2^n}{\rho}  \frac{n^2}{\sigma}  \lVert G_n \rVert_{\beta_n,\rho,\sigma_n,O_n}\lesssim 2^n n^2 \sigma^{-1} \rho^{-1}\varepsilon_n^{1/4}\ll 1.
\end{equation}
Hence,  it follows from  Lemma \ref {lem:bracketlemma} that $H^{(n+1)}\in \mathcal X_{\beta_{n+1},\tilde{\rho}_{n+1},\sigma_{n+1},O_{n+1}}$ and 
\[
\lVert R^{(n+1)}\rVert_{\beta_{n+1},\tilde{\rho}_{n+1},\sigma_{n+1},O_{n+1}}\leq  (1+\sum_{j=1}^{n+1}\varepsilon_j^{1/8}).
\]
It is also easy to check that (use \eqref{eq:smallchangeLag} and Lemma \ref{lem:bracketlemma})
\begin{equation}\label{eq:cauchyseq}
\lVert R^{(n+1)}-R^{(n)}\rVert_{  \beta_{n+1}\black,\tilde\rho_n,\sigma_n,O_n}\lesssim 2^n n^2 \sigma^{-1}\rho^{-1}\varepsilon_n^{1/4}.
\end{equation}
In particular $\{H^{(n)}_\xi\}_{n\in\mathbb N}$ forms a Cauchy sequence in  the Banach space $\mathcal Y_{\beta/2,\rho/2,\sigma/2,\mathcal O_F}$ so it has a limit $ H_\xi\in \mathcal Y_{\beta/2,\rho/2,\sigma/2,\mathcal O_F}$.   Note as well that, in view of the estimates \eqref{eq:smallchangeLag} and Lemma \ref{lem:bracketlemma}
the sequence of symplectic maps $\Phi^{(n)}=\phi_{G_1}\circ\cdots\circ \phi_{G_n}$ converges uniformly on $\mathcal Q_{\rho/2,\sigma/2}$ to a symplectic map $\Phi^\infty\in \mathcal H( \mathcal Q_{\rho/2,\sigma/2}, \mathcal Q_{\rho,\sigma}$) 
(recall the notation introduced in  \eqref{eq:phasespacefulldim} and \eqref{eq:Banachspacehol}).
\end{proof}

This finishes the proof of Theorem \ref{thm:mainlagrangianprecise} up to establishing that the set $\mathcal O_F=\bigcap_{n\in\mathbb N} \mathcal O_n$ is not empty. This will be shown to be true in Section \ref{sec:measureestimatesLag} provided $\varepsilon$ is sufficiently small compared to $|b-a|$. In particular, we will see that the condition $\varepsilon^{1/24}\leq |b-a|$ is sufficient. Combining this condition and the 
 one in  \eqref{eq:sigmarhosmallnessLag} we deduce that Theorem \ref{thm:mainlagrangianprecise} holds provided 
\begin{equation}\label{eq:finalsmallnessconditionlag}
\varepsilon\leq \min\{\rho^6\sigma^6,\ |b-a|^{24}\}.
\end{equation}
\begin{rem}
    No effort has been made to try to optimize the (very far from optimal)   condition \eqref{eq:finalsmallnessconditionlag}\black. \black
\end{rem}

\section{Measure estimates for the set $\mathcal O_F$}\label{sec:measureestimatesLag}

In this section we prove the part of Theorem \ref{thm:measestimatesgeneral} concerning the set $\mathcal O_F$ in Theorem \ref{thm:mainlagrangianprecise}. 

\subsection*{Construction of a suitable probability measure}\label{sec:clusteringfrequencies}

Take any $0< d<1$ and denote by $\mathcal C(d)\subset[a,b]^\mathbb N$ the set of frequencies for which $\mathrm{dim}_{box}(\{\xi_n\}_{n\in\mathbb N})= d$. It is around these frequencies where we will construct a suitable probability measure for which Diophantine vectors occuppy a large measure. 
\begin{rem}
It is an easy exercise to check that, for any $\gamma>0$ the sequence $\{n^{-\gamma}\}_{n\in\mathbb N}\subset[0,1]$ has box dimension $1/(1+\gamma)$.
\end{rem}
Given any $\ell>0$ and $\xi \in\mathcal C(d)$ we now define the set
\begin{equation}\label{eq:thickeningclusteredlag}
\Lambda(\xi)=\prod_{n\in\mathbb N} \bigg[\xi_n- \frac{\ell}{n^{\frac 1d}},\ \xi_n+  \frac{\ell}{n^{\frac 1d}}\bigg]
\end{equation}
\begin{rem}
One should think of $\ell$ as the room that one has to move the parameter $\xi$. 
\end{rem}
We endow $\Lambda(\xi)$ with the product topology. We recall that a cylinder is a set of the form 
\[
\prod_{n\in\mathbb N} A_n
\]
with $A_{n}\neq [\xi_n-\ell n^{-1/d}, \xi_n+\ell n^{-1/d}]$ for only finitely many indices and that the product topology is the one generated by open cylinder sets. We then consider the product $\sigma$-algebra $\mathcal A$ generated by the Borel cylinder sets (we require $A_{n}$ to be a  Borel set) and define the (unique) probability measure $\mu:\mathcal A\to\mathbb [0,1]$ such that, for any Borel cylinder set
\begin{equation}\label{eq:probmeasurelag}
\mu \left( \prod_{n\in\mathbb N} A_{n}\right)=\prod_{n\in\mathbb N} \frac{\mathrm{Leb} (A_{n})}{2\ell n^{-\frac1d}}.
\end{equation}

\subsection*{Measure estimates}
Given $\xi\in\mathcal C(d)$ we now estimate from below the $\mu$-measure of the set
\[
\mathcal O_F\cap \Lambda(\xi)\qquad\qquad\text{where}\qquad\qquad \mathcal O_F=\bigcap_{n\in\mathbb N}\mathcal O_n
\]
with the sets $\mathcal O_n$  as in Definition \eqref{defn:diophantinesetLag}. To that end, we  express
\begin{equation}\label{eq:setBnLag}
{\mathcal O}_n\cap\Lambda(\xi)=\Lambda(\xi)\setminus\left(\bigcup_{k=0}^{n} B_k(\xi)\right),\qquad\qquad
B_n(\xi)=\bigcup_{\underline{A}\in\mathbb{A}_n} \bigcup_{\substack{l\in\mathbb{Z}^{|A|}\\0<|l|\leq L_{n}}} \mathcal{R}^{(n)}_{A,l}(\xi)
\end{equation}
where  the (resonant) strips $\mathcal{R}^{(n)}_{A,l}(\xi)$ are defined as
\begin{equation}\label{eq:resonancestriplag}
\mathcal{R}^{(n)}_{A,l}(\xi)=\{\tilde\xi\in \Lambda(\xi)\colon |\langle \tilde\xi_A,l\rangle|\leq  \varepsilon_n^{1/12}\}.
\end{equation} 
We start by observing the following.
\begin{lem}
    For any $\xi\in [a,b]^\mathbb N$ the set $ \mathcal O_F\cap \Lambda(\xi)$ is $\mu$-measurable.
\end{lem}
\begin{proof}
    It is a straightforward consequence of the fact that  the resonant strips in \eqref{eq:resonancestriplag} are cylinder sets.
\end{proof}

We now estimate 
$\mu$-measure of each cylinder.

\begin{lem}\label{lem:measureresonantsetLag}
 Let $0<d<1$ and $\xi\in\mathcal C(d)\subset[a,b]^\mathbb N$. For any $n\in\mathbb N$, any $\underline A\in \mathbb A_n$ and any $l\in\mathbb Z^{|A|}\setminus\{0\}$ the set $\mathcal R_{A,l}^{(n)}(\xi)$ in \eqref{eq:resonancestriplag} satisfies
\begin{equation}\label{eq:measureresonancezonelag}
\mu (\mathcal{R}^{(n)}_{A,l}(\xi))\lesssim  \ell^{-1} (n^2|\log\varepsilon_n|)^{1/d}  \varepsilon_n^{1/12}.
\end{equation}
\end{lem}
\begin{proof}
Notice that, since $(a,b)\cap\{0\}=\emptyset$, by shrinking $\varepsilon_0$ if necessary, we may assume that 
\[
R_{A,l}^{(n)}(\xi)=\emptyset
\]
for all $l\in\mathbb Z^{|A|}$ with $|l|_1=1$. Therefore, if $R_{A,l}^{(n)}(\xi)\neq \emptyset$ we must have that $l$ is supported on a subset of $A$ with cardinality at least two.  On the other hand, since $A\in\mathbb A_n$, there must exist 
\[
j\leq \mu_n^{-1}|\log(\varepsilon_{n+1}/\varepsilon_n)|\lesssim n^2|\log\varepsilon_n|
\]
such that $\{j\}\in \mathrm{supp} (l)$. Let $v=\mathbbm 1_j$. Then,   
\begin{align*}
\frac{\mathrm{d}}{\mathrm{d}t}\left( \langle \tilde\xi_A+tv,l\rangle\right)=&l_j\neq 0.
\end{align*}
The proof now follows easily from the definition of the measure $\mu$ (see \eqref{eq:probmeasurelag})  and Fubini's theorem.
\end{proof}

We now proceed to estimate the measure of $B_n$ in \eqref{eq:setBnLag} and explain where the hypothesis that $\xi\in \mathcal C(d)$ with $d<1$ enters the picture. At a heuristic level, the idea is the following.  In \eqref{eq:setBnLag} we have expressed the set $B_n$ as union of measurable sets, i.e. the sets $\mathcal R^{(n)}_{A,l,s}$. In  Lemma \ref{lem:measureresonantsetLag} we have obtained an upper estimate for the measure of $\mathcal R^{(n)}_{A,l,s}$.
However, the union in \eqref{eq:setBnLag} comprises 
 countably many elements (notice that we take the union over $\underline A\in\mathbb A_n$), while the upper estimate in  Lemma \ref{lem:measureresonantsetLag} for the measure of the sets $\mathcal R^{(n)}_{A,l,s}$ is uniform. In the following lemma we show how the fact that $\xi\in \mathcal C(d)$ with $d<1$ allows us to replace the infinite union
\begin{equation}\label{eq:setBn2}
B_n=\bigcup_{A\in\mathbb{A}_n}\bigcup_{\overline A<j<\infty} \bigcup_{\substack{l\in\mathbb{Z}^{|A|+1}\\0<|l|\leq L_{n}}} \mathcal{R}^{(n)}_{A\cup\{j\},l}
\end{equation}
by a finite one. 

\begin{lem}\label{lem:inclusion}
Let $L_n$ be the constant defined in \eqref{eq:defninductive3}, let  $0<d<1$,  and define
\begin{equation}\label{eq:cutoffclustering}
q_n=\min\{q\in\mathbb N\colon L_n q^{^{-\frac 1d}}\leq    \varepsilon_n^{1/12}  \}.
\end{equation}
Then, there exists a subset $\mathcal I\subset \mathbb N_0$ with 
\[
\mathrm{card}(\mathcal I)\lesssim q_n
\]
such that, for any $A\in\mathbb A_n$ and any  $l\in\mathbb Z^{|A|+1}$ with $0<|l|_1\leq L_n$
\begin{equation}\label{eq:inclusionlag}
\bigcup_{\overline A< j<\infty} \mathcal R_{A\cup\{j\},l}^{(n)}\subset\left(\bigcup_{ j\leq q_n} \mathcal R^{(n)}_{A\cup\{j\},l}\right)\cup \left(\bigcup_{ j\in \mathcal I} \widehat{\mathcal R}_{A\cup\{j\},l}^{(n)}\right)
\end{equation}
where the set $\widehat {\mathcal R}_{A\cup\{j\},l}^{(n)}$ is defined as in \eqref{eq:resonancestriplag} but with the right hand side multiplied by $7$.
\end{lem} 
\begin{proof}
Since $\xi\in\mathcal C(d)$ there exists $C>0$ such that $\{\xi_q\}_{q\geq q_n}$ can be covered with less than $C {q_n}$ balls of size ${q_n}^{-\frac 1d}$. This implies that we can find a subset $\mathcal I\subset \mathbb N \cap\{i\geq q_n\}$ of cardinality less than $C {q_n}$
 such that $\{\xi_q\}_{q\geq q_n}$ is covered  with balls of size $2{q_n}^{-\frac 1d}$ centered at some element in  $\{\xi_i\}_{i\in\mathcal I}$. But then, for any $\tilde\xi\in \Lambda(\xi)$ and any $j\geq q_n$ there exists $i\in \mathcal I$ with 
\[
|\tilde\xi_j-\tilde\xi_i|\leq |\xi_j-\xi_i|+|\xi_j-\tilde\xi_j|+|\xi_i-\tilde\xi_i|\leq (2+2+2)q_{n}^{-\frac 1d}\leq 6\varepsilon_n^{1/12} L_n^{-1}
\]
Hence, if $\tilde\xi\in \mathcal{R}_{A\cup\{j\},l}^{(n)}$, then
\begin{align*}
 |\langle \tilde\xi_{A\cup\{i\}},l\rangle|\leq   |\langle \tilde\xi_{A\cup\{j\}},l\rangle|+|\langle (\tilde\xi_{A\cup\{i\}}-\tilde\xi_{A\cup\{j\}},l\rangle|
\leq \varepsilon_n^{1/12}+ |\tilde\xi_j-\tilde\xi_i| |l_i|
\leq 7 \varepsilon_n^{1/12}.
\end{align*}
Hence $\tilde\xi\in\widehat{\mathcal R}^{(n)}_{A\cup\{i\},l}$.
\end{proof}

In view of the inclusion \eqref{eq:inclusionlag} it follows from the estimate \eqref{eq:measureresonancezonelag} that 
\begin{equation}\label{eq:estimateb2lag}
\begin{split}
\mu (B_n)\leq & \sum_{A\in\mathbb{A}_n}\sum_{\substack{l\in\mathbb{Z}^{|A|+1}\\0<|l|_1\leq L_{n}}} \left(q_n+ \mathrm{card}(\mathcal I)\right) \ell^{-1} (n^2|\log\varepsilon_n|)^{1/d} \varepsilon_n^{1/12} \\
\leq &  \ell^{-1} (n^2|\log\varepsilon_n|)^{1/d} (\varepsilon_n^{1/12} )^{1-d}L_n^{d}\sum_{A\in\mathbb{A}_n} \sum_{\substack{l\in\mathbb{Z}^{|A|+1}\\0<|l|_1\leq L_{n}}} 1\\
\leq &  \ell^{-1} (n^2|\log\varepsilon_n|)^{1/d} (\varepsilon_n^{1/12} )^{1-d}L_n^{d}\sum_{A\in\mathbb{A}_n} L_n^{2|A|}.
\end{split}
\end{equation}
\begin{lem}\label{lem:technicalmeasuretechnicallag}
Let $\mathbb A_n$ and $L_n$ be defined in \eqref{eq:defninductive2} and  \eqref{eq:defninductive3} respectively. Then, 
\[
\sum_{A\in\mathbb{A}_n} L_n^{2|A|+1}\lesssim \exp(|\log\varepsilon_n|^{4/5}).
\]
\end{lem}

\begin{proof}
Define
  \[
  K_n=\frac{1}{\mu_n}|\log(\varepsilon_{n+1}/\varepsilon_n)|.
  \]
Thus $K_n\lesssim n^2 \log \varepsilon_n$. Observe that there are constants $R, \bar R$ such that 
$m_A\leq R^{|A|} e^{-\kappa |A|(|A|-1)/2}$
and $m_A\leq \bar R e^{-\kappa \bar A}.$
Hence if $A\in\mathbb A_n$ then $|A|\lesssim \sqrt{K_n}$ and $\overline A \lesssim K_n$. Therefore, 
  \[
  \sum_{ A\in\mathbb A_n} L_n^{2|A|}\lesssim \sum_{t=1}^{ C\sqrt{K_n}} \sum_{|A|=t} L_n^{2t}=
  \sum_{t=1}^{ C \sqrt{K_n}}\binom C K_n{t}L_n^{2t}\lesssim 
  \sum_{t=1}^{ C \sqrt{K_n}}   C^t K_n^t L_n^{2t}\lesssim ( C K_nL_n^2)^{C\sqrt{K_n}}.
  \]
Now the claim follows since the last expresion is less than $\exp|\log \varepsilon_n|^{0.51}$.
\end{proof}

Combining the upper estimate \eqref{eq:estimateb2lag} with the one obtained in Lemma \ref{lem:technicalmeasuretechnicallag} we deduce that
\begin{equation}\label{eq:finalestimateBn}
\mu(B_n)\lesssim \ell^{-1}  \varepsilon_n^{\frac{1}{24}(1-d)}.
\end{equation}
We have thus obtained the following result. Given $H_{\xi}$ as in \eqref{eq:normalformLag} we let $\mathcal O_F=\bigcap_{n} \mathcal O_n$ with $\mathcal O_n$  as in Definition \eqref{defn:diophantinesetLag}  for $\varepsilon_n$ arising from the inductive scheme discussed in Section \ref{sec:Kamlagrangian}.

\begin{thm}\label{thm:measureestimates}
Let $\{\xi_n\}_n\in\mathbb N\subset[a,b]^\mathbb N$ have box dimension $0<d<1$. Let $\Lambda(\xi)$ be defined in \eqref{eq:thickeningclusteredlag}, $\mathcal A$ be the product $\sigma$-algebra on $\Lambda(\xi)$ and $\mu$ the probability measure in $\mathcal A$ defined by \eqref{eq:probmeasurelag}. Suppose that $H_{\xi}$ and $\mathcal O_{F}$ are as above. Then, 
\[
\mu(\mathcal O_F\cap \Lambda(\xi))\geq 1-O(\ell^{-1} \varepsilon ^{\frac{1}{24}(1-d)}).
\]
\end{thm}

\section{Systems of interacting particles: Proof of Theorems \ref{thm:mainLagrangianinformal} and \ref{thm:mainellipticpts}}\label{sec:actionanglenormalform}
In Section \ref{sec:resultsnormalforms} we show how to establish Theorems \ref{thm:mainLagrangianinformal} and \ref{thm:mainellipticpts} using Theorem \ref{thm:mainlagrangianprecise}. Then,  in Section \ref{sec:abundanceKAM} we discuss the abundance of the collection of tori in Theorems \ref{thm:mainLagrangianinformal} and \ref{thm:mainellipticpts}.

\subsection{Proof of Theorem \ref{thm:mainLagrangianinformal}}\label{sec:resultsnormalforms}

In this section we show how to reduce non-degenerate class $\mathcal H$ Hamiltonians as in  \eqref{eq:firstHamnew}-\eqref{eq:firstHamnew2} to the normal form \eqref{eq:normalformLag} to obtain a proof of Theorem \ref{thm:mainLagrangianinformal}. The proof of Theorem \ref{thm:mainellipticpts} is deduced by a rather similar argument; hence, it is left to the reader.

\subsection*{Abstract results in higher dimensions}\label{sec:generalN}

In Section \ref{sec:mainabstract}, we have only presented the abstract theory for $(\varphi,J)\in \mathcal Q_{\rho,\sigma}$. However, as we will see below, the normal form associated to the mechanical Hamiltonian \eqref{eq:firstHam} corresponds to  a real-analytic Hamiltonian $H_\xi:\mathcal Q_{\rho,\sigma}^N\to \mathbb C$.

The exposition in Section \ref{sec:Kamlagrangian} readily generalizes to the general case $N\in\mathbb N$ and the corresponding main results in Theorem \ref{thm:mainlagrangianprecise}  hold unchanged. The measure estimate in Theorem \ref{thm:measestimatesgeneral} also holds unchanged for $\{\tilde \xi_n\}_{n\in\mathbb N}=\{(\tilde \xi_{n,1},\dots,\tilde\xi_{n,N})\}_{n\in\mathbb N}$ with (we interpret $\{\tilde \xi_n\}_{n\in\mathbb N}$ as a subset of $\mathbb R^N$) $\mathrm{dim}_{box}(\tilde\xi)=d\in(0,1)$.

\subsection*{From class-$\mathcal H$ Hamiltonians to normal forms}

We recall that class-$\mathcal H$ Hamiltonians were introduced in Definition \ref{defn:Nbodytype1}. 
\medskip

We recall the splitting \eqref{eq:firstHamnew}  (we relabel $\varepsilon$ as $\epsilon$ since we want to keep the symbol $\varepsilon$ for later\black)
\[
H(x,y)=H_{0}(x,y)+\epsilon H_1(y,x)
\]
where
\[
H_{0}(x,y)=\sum_{i\in\mathbb N}m_i h_i(x_i,y_i/m_i)\qquad\qquad H_1(x)=\sum_{i<j}m_i m_j V_{ij}(x_i,x_j,y_i/m_i,y_j/m_j)
\]
to highlight the perturbative setting. 
We now introduce the conformally symplectic scaling $\phi_{\bs m}:(x,\tilde y)\mapsto (x,y)$ given by 
\[
y_i=m_i \tilde y_i.
\]
In the scaled coordinates $(x,\tilde y)\in \ell_\infty^N\times \ell^N_\infty$ the Hamiltonian reads 
\[
H\circ\phi_{\bs m}(x,\tilde y)=\sum_{i\in\mathbb N} m_i h_{i}(x_i,\tilde y_i)+\epsilon \sum_{i<j} m_im_j \widetilde V_{ij}(x_i,x_j,\tilde y_i,\tilde y_j;m_i,m_j)
\]
 and the symplectic structure is given by
\[
(\phi_{\bs m})^*\left(\sum_{i=0}^\infty \mathrm{d} y_i\wedge\mathrm{d} x_i \right)= \sum_{i=0}^\infty m_i \mathrm{d}\tilde y_i\wedge\mathrm{d} x_i.
\]
Suppose now that $\{h_i\}_{i\in\mathbb N}$ satisfies property \textbf{P1}\black. Then, by definition, there exists $\sigma_0,\rho_0,\gamma>0$ and $\tau\geq N$, a compact set $K\subset\mathbb R^N$, a   Diophantine  frequency vector  $\xi_0^i\in K\subset \mathbb R^N$ of the class $DC(\gamma,\tau)$, \black  invertible matrices $A_i$ and a real-analytic, symplectic transformation (here we use the notation $\mathbb B_\rho\subset\mathbb C^N$ for the complex ball of radius $\rho$ and $\mathbb T_\sigma\subset (\mathbb C/2\pi\mathbb Z)^N$ for the complex torus of width $\sigma$)
\begin{equation}\label{eq:psisect4}
\psi_i:(\varphi_i,J_i)\in \mathbb T_{\sigma_0}\times \mathbb B_{\rho_0}\mapsto (x_i,y_i)\in \mathbb C^N\times\mathbb C^N
\end{equation}
such that 
\[
N_{i}(\varphi_i,I_i):=h_{i}\circ\psi_i(\varphi_i,I_i))=\langle\xi_0^i,I_i\rangle+\frac12 \langle A_iI_i,I_i\rangle+P_i(\varphi_i,I_i)
\]
for some $P_i=O_3(I_i)$.   We now flatten the normal form $N_{i}$ to high order. Indeed, since $\xi_0^i\in DC(\gamma,\tau)$, using a standard argument, one can find $\rho,\sigma>0$ (depending only on $\rho_0,\sigma_0$ and $\gamma,\tau$) and a real-analytic, close to identity  map $\hat\psi_i:\mathbb T_{\sigma}\times \mathbb B_{\rho}\to \mathbb T_{\sigma_0}\times \mathbb B_{\rho_0}$  which coincides with the identity on the torus $\{I=0\}$ and such that 
\begin{equation}\label{eq:bnfkolmogorovtorus}
N_{i}\circ\hat\psi_i(\varphi_i,I_i)=\langle \xi_0^i,I_i\rangle +\hat h_i(I_i)+ \hat P_i(\varphi_i,I_i)
\end{equation}
for
\[
\hat h_i(I_i)=\frac 12 \langle A_iI_i,I_i\rangle+O_{3} (I_i)
\]
and $\partial_{I_i}^k \hat P_i|_{\{I_i=0\}}=0$ for any $0\leq k\leq 100$.  Since $A_i$ is invertible, the map  
\[
I_i\mapsto \xi_0^i+\partial_{I_i} h^{(n)}_i(I_i)
\]
is a diffeomorphism for $I_i\in \mathbb B_{\rho}$ (after shrinking, if necessary, the value of $\rho$)\black. Thus, if we denote by $C_i\subset \mathbb C^N$ the set 
\[
C_i=\xi_0^i+\partial_{I_i} h_i(\mathbb B_{\rho_\epsilon}) \qquad\qquad \text{with }\rho_\epsilon=\epsilon^{1/100},
\]
for $\xi_i\in C_i$ we can consider the inverse map 
\begin{equation}\label{eq:freqmappp}
\xi_i\mapsto I_i(\xi)=A_i^{-1}(\xi_i-\xi_0^i)+O_2(|\xi_i-\xi_0^i|).
\end{equation}
\black and define the change of variables 
\begin{equation}\label{eq:psiomegasect4}
\psi_{\xi_i}:(\varphi_i,J_i)\mapsto (\varphi_i, I_i(\xi)+J_i).
\end{equation}
  We can then introduce the $N$-parametric family of Hamiltonians (parametrized by $\xi_i\in C_i\subset\mathbb C^N$)
\[
N_{\xi_i,i}=N_{i}\circ\hat\psi_i\circ\psi_{\xi_i}(\varphi_i,J_i)=\langle\xi,J_i\rangle+ P_{h,i}(\varphi_i,J_i;\xi_i)+\widetilde P_i(\varphi_i,J_i;\xi_i)\qquad\qquad (\varphi_i,J_i)\in \mathbb T_{\sigma}\times\mathbb B_{\rho_\epsilon}
\]
 with $P_{h,i}$ satisfying $\partial_{I_i} ^k P_{h,i}|_{\{I=0\}}=0$ for $0\leq k\leq 1$ and 
 \[
 \sup_{(\varphi_i,J_i)\in \mathbb T_{\sigma}\times\mathbb B_{\rho_\epsilon}} |\widetilde P_i(\varphi_i,J_i;\xi_i)|\lesssim \epsilon. 
 \]
Finally, we consider any 
 \begin{equation}\label{eq:Cset}
 \xi\in \mathcal C:= \prod_{i\in\mathbb N}C_i
 \end{equation}
 and define the real-analytic change of variables 
 \begin{equation}\label{eq:compossition}
 \Theta_{\xi,\bs m}=\phi_{\bs m}\circ\Psi\circ\hat\Psi\circ\Psi_\xi:(\varphi,J)\in \mathcal Q^N_{\rho_\epsilon,\sigma}\to \ell_\infty^N\times\ell_{\infty,\bs m}^N 
 \end{equation}
where $\Psi,\Psi^{(n)},\Psi_\xi$ are the coordinate transformations which for each $i\in \mathbb N$ act as in \eqref{eq:psisect4}, \eqref{eq:bnfkolmogorovtorus} and \eqref{eq:psiomegasect4}. 
\medskip

We now conclude the proof of Theorem \ref{thm:mainLagrangianinformal} as follows. We fix any $\tilde\xi\in \mathcal C$ with\footnote{ Each $C_i$ contains of a cube of size $\rho_\epsilon=\epsilon^{1/100}$ around $\xi_0^i$. The existence of the desired $\tilde\xi$ is then a trivial consequence of the fact that, by assumption, there exists a compact $K\subset\mathbb R^N$ such that $\{\xi_i\}_{i\in\mathbb N}\subset K$.} 
 \[
    \mathrm{dim}_{\mathrm{box}}(\{\tilde\xi_i\}_{i\in\mathbb N})\in (0,1)
    \]
    and such that the set 
\[
O=\{\xi\in \ell_\infty\colon |\xi-\tilde\xi|_\infty\leq \epsilon^{1/25}\}.
\]
is contained in $\mathcal C$. It is straightforward to check that, for any $\beta\in (0,1)$:  the Hamiltonian $H\circ\Theta_{\xi,\bs m}$ belongs to the Banach space $\mathcal Y_{\beta,\rho_\epsilon,\sigma,O}$ introduced in Section \ref{sec:definitions},  is of the form 
\[
H_\xi(\varphi,J)=\xi\cdot H+P(\varphi,J;\xi)
\]
with $P=P_h+\widetilde P$
\begin{align*}
P_h=&\sum_{i} m_i P_{h,i}  \qquad\qquad 
\widetilde P=\sum_{i} m_i \widetilde P_{i}+\epsilon\sum_{i<j} m_im_j \widehat V_{ij} 
\end{align*}
for $\widehat V_{ij}=\widetilde V_{ij}\circ(\psi_i,\psi_j)\circ(\hat\psi_i,\hat\psi_j)\circ(\psi_{\xi_i},\psi_{\xi_j})$ and the symplectic form is given by $
\sum_i m_i \mathrm{d}J_i\wedge\mathrm{d}\varphi_i$. Note moreover that, for any $\beta<1$
\[
\lVert \widetilde P\rVert_{\beta,\rho,\sigma,\mathcal C}\lesssim \epsilon +\frac{1}{1-\beta}\epsilon=:\varepsilon
\]
Hence, since $\sigma$ is independent of $\epsilon$, $\rho_\epsilon=\epsilon^{1/100}$, and the ``width'' of the parameter domain $O$ is $\epsilon^{1/25}$, 
  the smallness condition \eqref{eq:finalsmallnessconditionlag} needed to apply Theorem \ref{thm:mainlagrangianprecise} is met (provided $\epsilon>0$ is small enough).  Hence, the proof of Theorem \ref{thm:mainLagrangianinformal} follows from a higher dimensional version of Theorem \ref{thm:mainlagrangianprecise}, which, as discussed at the beginning of this section, can be proved in the very same way as Theorem \ref{thm:mainlagrangianprecise}.

\subsection{(Singular) set of KAM tori}\label{sec:abundanceKAM}
In this short section we describe where the KAM set sits inside the phase space\black. As we will make use of notation in Theorem \ref{thm:mainlagrangianprecise}, which is only stated for $N=1$, we only discuss this case in order to avoid introducing extra notation. Let  $\mathcal C$ be as in \eqref{eq:Cset} and let
\[
\Omega:\mathcal C\to \ell_\infty
\]
be the map which acts on each $i\in\mathbb N$ by \eqref{eq:freqmappp} and let $\Phi_{\xi,\bs m}=(\Phi^\varphi_{\xi,\bs m},\Phi^J_{\xi,\bs m})$ be as in \eqref{eq:compossition}. Let $\mathcal O_F\subset \mathcal C$ be the set of frequencies produced by Theorem \ref{thm:mainlagrangianprecise} for the Hamiltonian $H\circ\Phi_{\xi,\bs m}$ constructed above. Let $\Phi$ and $\phi$ be the maps constructed in Theorem \ref{thm:mainlagrangianprecise} for the Hamiltonian $H\circ\Phi_{\xi,\bs m}$.  Define the map 
\begin{align*}
\mathcal F:\mathcal T^\infty\times \mathcal O_F&\to \ell_\infty\times\ell_{\infty,\bs m}\\
(\varphi,\xi)&\mapsto (\Phi^\varphi_{\phi(\xi),\bs m}(\varphi,0), \Omega(\phi(\xi)))
\end{align*}
Then, the KAM set for which Theorem \ref{thm:mainLagrangianinformal} holds is given by 
\[
\mathrm{KAM}=\mathcal F(\mathcal T^\infty\times \mathcal O_F).
\]

\section{Applications to systems of interacting particles} \label{sec:particles}

We now present applications of our theory to systems of mechanical particles: 
given $N\in\mathbb N$, we consider an (a priori formal) Hamiltonian $H:\mathbb R^{2N\mathbb N}\to \mathbb R\cup\{\pm\infty\}$ of $n$-body type
\begin{equation}\label{eq:firstHam1}
H(x,y)=\sum_{i=0}^\infty \frac{|y_i|^2}{2\mu_i}-\sum_{i<j}\mu_i\mu_j V(x_i-x_j)\qquad\qquad (x_i,y_i)\in\mathbb R^N\times\mathbb R^N,
\end{equation}
for some real-analytic potential $V$ and mass vector $\{\mu_i\}_{i=0}^\infty\in\ell_\infty^\mathbb R$. The Hamiltonian \eqref{eq:firstHam1} represents the energy of a system composed of an infinite number of particles moving in $N$-dimensional Euclidean space and interacting through the potential $V$. The coordinates $x_i$ and $y_i$ denote, respectively, the position and momentum of particle $i$.  

In order to view \eqref{eq:firstHam1} as an infinite-dimensional dynamical system, we must select a suitable phase space. Observe first that one expects to encounter a wide variety of dynamical behaviors in different regions of the parameter space (i.e., the masses of the bodies) and/or different choices of phase spaces. In this paper, we restrict our attention to the following setting:
\begin{itemize}
\item \textit{(Planetary regime):} We fix  the masses
\begin{equation}\label{eq:mu}
\mu_0=1,\qquad\qquad \bs\mu:=\{\mu_i\}_{i\in\mathbb N}=\{\varepsilon m_i\}_{i\in\mathbb N},
\end{equation}
with $\bs m\in \ell^\mathbb R_{\infty,\kappa}$ for some $\kappa>0$ (recall that $\ell^\mathbb R_{\infty,\kappa}$ was defined in \eqref{eq:spacedecayingmasses}). Namely, we are interested in the region of the parameter space known in the Celestial Mechanics literature as the \textit{planetary regime}: one heavy body interacts with a collection of much lighter bodies (planets). We will see later that, under this choice of parameters, the Hamiltonian $H$ can be expressed as the sum of two terms: the first being an infinite sum of uncoupled (a priori non-integrable) two-body Hamiltonians, describing the interaction of each body with the heavy one; and the second, a term of smaller order, describing the all-to-all interactions among the light bodies.
\\

\item \textit{(Bounded-energy):} We further focus on the \textit{bounded-energy region}
\[
(x,y)\in\ell_\infty^N\times\ell_{\infty,\bs m}^N,\qquad\qquad 
\ell_{\infty,\bs m}=\{z\in\ell_\infty:\ \sup_{i\in\mathbb N} m_i^{-1}|z_i|<\infty\}.
\]
In physical terms, configurations belonging to this phase space satisfy, for any $i\in\mathbb N$,
\[
\frac{E_i(x_i,y_i)}{m_i}\sim 1,\qquad\qquad 
E_i(x_i,y_i)=\frac{|y_i|^2}{2m_i}-m_iV(|x_i-x_0|),
\]
that is, the ratio between the energy and the mass of each particle is of order one (i.e., the velocities are bounded). Consequently, the total energy of the system is finite. This choice of the phase space is again motivated by Celestial Mechanics. In particular, in this regime we will construct orbits in which all the light particles rotate at a uniform distance around the heavy one (as in our solar system).
\end{itemize}

In this setting, the long range nature of the problem already manifests itself at the level of the equations of motion. Indeed, it follows from Newton's equations associated with the Hamiltonian \eqref{eq:firstHam1} that, for any $i\in\mathbb N$,
\begin{equation}\label{eq:newton}
\mu_i \ddot x_i=\sum_{i\neq j}\mu_i\mu_j \nabla_{x_i} V(x_i,x_j),
\end{equation}
so the mass $\mu_i$ cancels on both sides. Because of the all-to-all interaction, all particles except the first one experience accelerations of comparable magnitude.\footnote{Which imply magnitude of order $m_i$ on the momentum of particle $i$, and this is macroscopic given the normalization that we adopt.} Thus, these systems constitute natural candidates for testing our abstract results presented above. To do so, we introduce the following class of Hamiltonians.

\begin{defn}\label{defn:Nbodytype}
Let $N\in\mathbb N$ and $\bs m\in\ell^\mathbb R_{\infty,\kappa}$ for some $\kappa>0$. We say that 
\[
H(x,y):\ell_{\infty}^N\times\ell_{\infty,\bs m}^N\to\mathbb C
\]
is a \textit{mechanical Hamiltonian} if it is of the form
\begin{equation}\label{eq:firstHam}
H(x,y)=\sum_{i=1}^\infty  \left(\frac{|y_i|^2}{2m_i}-m_i V(x_i)\right)
+\varepsilon\sum_{i<j} m_im_j V_{i,j}(x_i,x_j,y_i/m_i,y_j/m_j)
\end{equation}
and there exists $K>0$ such that $V\in C^\omega((-K,K)^{N},\mathbb R)$ and 
$V_{i,j}\in C^\omega((-K,K)^{4N},\mathbb R)$ are  uniformly  bounded on some compact complex neighbourhood of $(-K,K)^N$ or $(-K,K)^{4N}$.
\end{defn}

The definition above calls for a couple of remarks:
  \begin{itemize}
  \item In the planetary regime, after introducing heliocentric coordinates (see Lemma  \ref{lem:heliocentric}), systems of the form \eqref{eq:firstHam1} recast as mechanical Hamiltonians as in \eqref{eq:firstHam} with 
   \begin{equation}\label{eq:potentialheoliocentric}
   V_{i,j}= \frac{\langle y_i,y_j\rangle}{m_im_j}-V(x_i-x_j).
   \end{equation}
   \item Another physically relevant example is given by Hamiltonians
   \begin{equation}\label{eq:background}
H(x,y)=\sum_{i=1}^\infty  \left(\frac{|y_i|^2}{2m_i}-m_i V(x_i)\right)
-\varepsilon\sum_{i<j} m_im_j V_{i,j}(x_i,x_j),
\end{equation}
which model the motion of infinitely many particles moving in a \textit{background potential} $V$ and which interact via couplings of the form $m_im_jV_{i,j}$. 

\item Mechanical Hamiltonians as in Definition \ref{defn:Nbodytype} can be seen as a subset of the class-$\mathcal H$ Hamiltonians which were introduced in Definition \ref{defn:Nbodytype1}.
\end{itemize}

\begin{lem}\label{lem:heliocentric}
Let $N\in\mathbb N$, let $\bs m\in\ell^\mathbb R_{\infty,\kappa}$ for some $\kappa>0$, and let $H:(x,y)\in \ell_\infty^2\times\ell_{\infty,\bs m}^2\to \mathbb R$ be as \eqref{eq:firstHam1} (where $\bs \mu$ is as in \eqref{eq:mu}).  There exists a holomorphic (conformally) symplectic map $\phi_{\mathrm{hel}}:(x,y)\in \ell_\infty^2\times\ell_{\infty,\bs m}^2 \mapsto (X,Y)\in \ell_\infty^2\times\ell_{\infty,\bs m}^2$ such that after rescaling time by the constant $\varepsilon>0$ conjugates $H$ to a mechanical Hamiltonian \eqref{eq:firstHam} with $\{V_{i,j}\}_{i,j}$ as in \eqref{eq:potentialheoliocentric}.
\end{lem}
Although this a very well-known fact (see for instance \cite{AKN}), we present the proof here for the sake of self-completeness.
\begin{proof}
     After rescaling the momentum vector and time by $\varepsilon$ we obtain that
\[
H(x,y)=H_{\mathrm{\mathrm{mech}}}(x,y)-\varepsilon F(x)
\]
where
\[
H_{\mathrm{\mathrm{mech}}}(x,y)=\frac{|y_0|^2}{2\varepsilon}+\sum_{i\in\mathbb N}\left(\frac{|y_i|^2}{2m_i}-m_i V(|x_i-x_0|)\right)\qquad\qquad F(x)=\varepsilon\sum_{i<j}m_im_j V_{ij}(x_i,x_j)
\]
so the perturbative setting becomes apparent. Introduce now the  change of coordinates $\phi_{\mathrm{hel}}:(x,y)\in \ell_\infty^N\times\ell_{\infty,\bs m}^N \mapsto (X,Y)\in \ell_\infty^N\times\ell_{\infty,\bs m}^N $ given by
\begin{equation}\label{eq:heliocentric}
X_i=x_i-x_0\qquad\qquad Y_0=\sum_{i=0}^\infty y_i\qquad\qquad Y_i= y_i.
\end{equation}
It is then a trivial computation to check that in   \textit{heliocentric coordinates} $(X,Y)$, (after a time rescaling)
\[
H\circ \phi_{\mathrm{hel}}(X,Y)=\sum_{i\in\mathbb N} h_{\mathrm{\mathrm{mech}}}(X_i,Y_i)+\varepsilon\sum_{j<k} \langle Y_j,Y_k\rangle -\varepsilon\sum_{i<j}m_im_j V_ij(X_i,X_j)
\]
and
\[
(\phi_{\mathrm{hel}})^*\left( \sum_{i=0}^\infty \mathrm{d}Y_i\wedge\mathrm{d} X_i\right)=\sum_{i=0}^\infty \mathrm{d}y_i\wedge\mathrm{d} x_i.
\]\qedhere
\end{proof}

\subsection{Main results: applications to mechanical systems}
We now present our main results in the context of mechanical systems of interacting particles. These boil down to particular cases of Theorems \ref{thm:mainLagrangianinformal}, \ref{thm:mainellipticpts} under suitable non-degeneracy conditions on 
\begin{equation}\label{eq:keplerintro}
    h_{\mathrm{mech}}(x,y)=\frac{|y|^2}{2}-V(x)
\end{equation}
\subsection*{Full-dimensional tori}
The following result follows trivially as a particular case of Theorem \ref{thm:mainLagrangianinformal}.
\begin{thm}[Full-dimensional KAM tori for interacting particles]\label{thm:mainLagrangianinformalparticles}
Let $N\in\mathbb N$, let $\bs m\in\ell^\mathbb R_{\infty,\kappa}$ for some $\kappa>0$, and let $H:\ell_\infty^N\times\ell_{\infty,\bs m}^N\to\mathbb C$ be a mechanical Hamiltonian with $h_{\mathrm{mech}}$ satisfying \textbf{P1}. Then there exists a constant $\varepsilon_0(V,\kappa)>0$ such that for any $0\leq \varepsilon\leq \varepsilon_0(V,\kappa)$ there exists an uncountable collection of real-analytic full-dimensional tori near the infinite dimensional torus given by \textbf{P1} for $\varepsilon=0$, invariant under the flow of $H$, on which the motion is conjugated to a linear translation on $\mathbb T^{N\mathbb N}$.

\end{thm}

 A particularly relevant case of this setting corresponds to $V=V(|x|)$. In this case,
\[    h_{\mathrm{mech}}(x,y)=\frac{|y|^2}{2}-V(|x|)
\]
is integrable (integrability follows from rotational invariance), and its phase space contains large regions foliated by (full) $N$-dimensional invariant tori on which the motion is conjugated to a linear translation on the torus. Moreover, under an explicit non-degeneracy condition on $V$,   for $N=1$ (motion on a line) or $N=2$ (motion on the plane)\black, the map assigning to each  torus its corresponding frequency vector is a local diffeomorphism so property \textbf{P1} holds for all tori with Diophantine frequency (see Section \ref{SSCentral} for a detailed analysis of the integrable case).

\medskip

\subsection*{KAM annuli around elliptic fixed points}
The following result constitutes a particular case of Theorem \ref{thm:mainellipticpts}.

\begin{thm}[KAM annuli for interacting particles]\label{thm:mainellipticptsparticles}
Let $N\in\mathbb N$, let $\bs m\in\ell^\mathbb R_{\infty,\kappa}$ for some $\kappa>0$, and let $H:\ell_\infty^N\times\ell_{\infty,\bs m}^N\to\mathbb C$ be a mechanical Hamiltonian with $h_{\mathrm{mech}}$ satisfying property \textbf{P2}. Then there exists a constant $\varepsilon_0(V,\kappa)>0$ such that for any $0<\varepsilon\leq \varepsilon_0(V,\kappa)$ there exists an uncountable collection of full-dimensional tori, invariant under the flow of $H$, on which the motion is conjugated to a linear translation on $\mathbb T^{N\mathbb N}$. These tori are localized around the   elliptic fixed point  given by \textbf{P2} for $\varepsilon=0$.\black
\end{thm}

\subsection{The integrable case}
\label{SSCentral}

The purpose of this section is to give a rather descriptive picture of how our theory fits into the integrable case $V=V(|x|)$. In particular we  recall that, under explicit conditions on $V$, the corresponding (integrable) Hamiltonian $h_{\mathrm{mech}}(x,y)=\frac{|y|^2}{2}-V(|x|)$ displays a foliation of (a large part of) its phase space by non-degenerate invariant tori. The section is organized as follows:
\begin{itemize}
    \item In Section \ref{secA1A2}, we show that under a suitable non-degeneracy assumption on $V$ (see Assumptions \textbf{A.1} and \textbf{A.2}) most of the phase space of $h_{\mathrm{mech}}$ is foliated by $N$-dimensional non-degenerate invariant tori. Moreover, around any of these tori we construct a local, real-analytic, change of coordinates which brings $h_{\mathrm{mech}}$ into the parametric normal form \eqref{eq:normalformP1}
   
\item In Section \ref{sec:examples}, we illustrate the discussion with some figures drawn from the concrete example of the Duffing Oscillator.
\end{itemize}
In both Sections \ref{secA1A2} and \ref{sec:examples}, we  will distinguish between the cases $N=1$ and $N=2$.
\vspace{0.2cm}

\subsubsection{Non-degenerate integrability} $ \ $ \label{secA1A2}
\black

\noindent\textbf{Case $N=1$:}
We consider the one degree-of-freedom mechanical Hamiltonian 
\begin{equation}\label{eq:keplern1}
h_{\mathrm{mech}}(r,y)=\frac{y^2}{2}-V(r)\qquad\qquad (r,y)\in\mathbb R_+\times\mathbb R.
\end{equation}

\begin{defn}[Assumption \textbf{A.1}]\label{defn:nondegeneracyN1}
We say that a potential $V$ satisfies assumption \textbf{A.1} if there exists a non-empty interval $(a,b)\subset\mathbb R$ such that for all $h\in (a,b)$:
\begin{itemize}
    \item (Compactness:)  the energy levels $\{h_{\mathrm{mech}}=h\}$ are smooth closed curves, 
    \item (Non-degeneracy:) the action map $h\mapsto L$ 
    \[
    L(h)=\frac{1}{2\pi} \int_{\{h_{\mathrm{mech}}(y,r)=h\}}y\mathrm{d}r
    \]
    admits an inverse map $L\mapsto h(L)$ which satisfies 
    \[
    \partial^2_{L^2} h\neq 0.
    \]
\end{itemize}
\end{defn}

The non-degeneracy condition in Definition \ref{defn:nondegeneracyN1} guarantees that the energy levels $\{h_{\mathrm{mech}}=h\}$ for $h\in (a,b)$ are periodic orbits and that, moreover, the period of these orbits is a non-trivial function of the energy.

\vspace{0.2cm}

\noindent\textbf{Case $N=2$:} In this case, after introducing polar coordinates ($\mathcal R$, the radial momentum, stands for the variable conjugated to $r$ and $\mathcal G$, the angular momentum, is the variable conjugated to $\alpha$)
\begin{equation}\label{eq:changepolar}
x=(r\cos\alpha,r\sin\alpha)\qquad\qquad y=(\mathcal R\cos\alpha-\frac{\mathcal G}{r}\sin\alpha,\mathcal R\sin\alpha+\frac{\mathcal G}{r}\cos\alpha),
\end{equation}
the Hamiltonian $h=\frac{|y|^2}{2}-V(|x|)$ recasts as 
\begin{equation}\label{eq:uncoupledpolar}
\mathtt h_{\mathrm{mech}}(r,\mathcal R;\mathcal G)=\frac{\mathcal R^2}{2}+\frac {\mathcal G^2}{2r^2}-V(r)\qquad\qquad (r,\mathcal R)\in\mathbb R_+\times\mathbb R.
\end{equation}
Notice that the angle $\alpha\in\mathbb T$ is a cyclic variable so the angular momentum $\mathcal G$ is a conserved quantity.

\begin{defn}[Assumption \textbf{A.2}]\label{defn:nondegeneracy}
We say that a real-analytic potential $V:(0,\infty)\to\mathbb R$ satisfies assumption \textbf{A.2} if there exists non-empty intervals $(a,b)\subset\mathbb R$ and $(g_1,g_2)\subset\mathbb R$ such that for any $(h,\mathcal G)\in (a,b)\times(g_1,g_2)$:
\begin{itemize}
    \item (Compactness:)  the level curves $\{(r,\mathcal R)\colon \mathtt h_{\mathrm{mech}}(r,\mathcal R;\mathcal G)=h\}$ are closed curves.
\item (Non-degeneracy:) the action map $(h,\mathcal G)\mapsto (L,\mathcal G)$ where  \[
L(h,\mathcal G)=\frac{1}{2\pi}\int_{\{\mathtt h_{\mathrm{mech}}(r,\mathcal R;\mathcal G)=h\}} \mathcal R\mathrm{d}r
\]
admits a local inverse $\Theta:(L,\mathcal G)\mapsto (h,\mathcal G)$ and $\hat h(L,\mathcal G)= \pi_h\Theta(L,\mathcal G)$ is such that $(L,\mathcal G)\mapsto \nabla \hat h(L,\mathcal G)$ is a local diffeomorphism.
\end{itemize}
\end{defn}
As in the case $N=1$,  the non-degeneracy condition in Definition \ref{defn:nondegeneracy} guarantees the existence of a local foliation by  full-dimensional invariant tori  on which, as we will see below,  the frequency map is non-degenerate (so these tori might be parametrized in terms of their  frequency vector).  We refer to these tori as \textit{Lagrangian}.

 We now introduce a local coordinate system  in the neighbourhood of any Lagrangian tori which brings $h_{\mathrm{mech}}$ into action-angle normal form with non-vanishing twist as in \eqref{eq:normalformP1}. We only deal with the case $N=2$, the case $N=1$ being similar.  We define the effective potential 
 \[
 V_{eff}(r;\mathcal G)=\frac {\mathcal G^2}{2r^2}-V(r)
 \]
 and introduce the generating function
\[
W(\alpha,r,\mathcal G,L)=\alpha \mathcal G+\int_a^r \sqrt{\hat h(L,\mathcal G)-V_{eff}(s;\mathcal G)}\mathrm{d}s
\]
where $a(L,\mathcal G)$ is the largest solution to $V_{eff}(r;\mathcal G)=-\hat h(L,\mathcal G)$. Hence, the canonical change of variables $(\varphi_1,L,\varphi_2,\mathcal G)\mapsto (r,\mathcal R,\alpha,\mathcal G)$ given by 
\begin{equation}\label{eq:actionangle}
\varphi_1=\partial_{L} W\qquad\qquad \varphi_2=\partial_{\mathcal G} W\qquad\qquad \mathcal G=\partial_\alpha W\qquad\qquad \mathcal R=\partial_r W
\end{equation}
converts the Hamiltonian $\mathtt h_{\mathrm{mech}}$ into the Hamiltonian 
\[
\mathtt H(L,\mathcal G)=\hat h(L,\mathcal G).
\]
  Note that Assumption \textbf{A.2} implies that the map $(L,\mathcal G)\mapsto \nabla \hat h(L,\mathcal G)$ is a local diffeomorphism. In particular, there exist $(L_*,\mathcal G_*)\in (a,b)\times (g_1,g_2)$ such that the vector 
\[
\xi_*=\nabla \hat h(L_*,\mathcal G_*)
\]
is Diophantine.  Introducing the change of variables 
\[
\psi:(J_1,J_2)\mapsto (L_*+J_1,\mathcal G_*+J_2)
\]
we obtain 
\[
\mathtt H\circ \psi(J)=\langle \xi_*,J\rangle +P(J)
\]
 with $\partial_{J}P|_{\{J=0\}}=0$.
 
 In conclusion, we have shown the following that allows to apply Theorem \ref{thm:mainLagrangianinformal}
 to these classes of integrable mechanical systems. 
 
\begin{prop}\label{lem:uncoupledchangevarsLag}
Let $N=1$ (resp. $N=2$) and suppose that that $V$ satisfies \textbf{A.1} (resp. \textbf{A.2}).  Then, $h_{\mathrm{mech}}(x,y)=\frac{|y|^2}{2}-V(|x|)$ satisfies \textbf{P1} for all Diophantine frequencies.
\end{prop}
\begin{figure}
\begin{center}
 \includegraphics[scale=0.45]{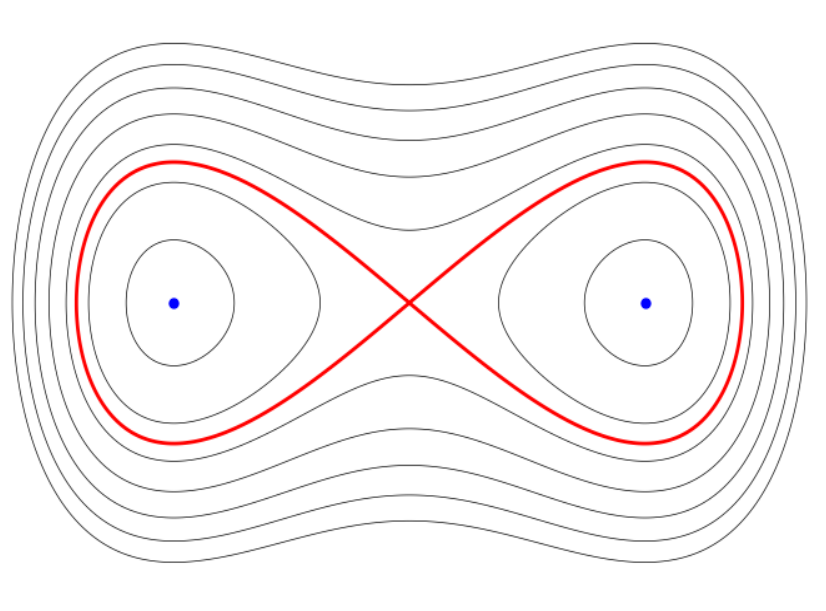}
   \caption{The Duffing Hamiltonian for $N=1$. The phase space is decomposed into five parts: An open unbounded component foliated by closed curves, a figure-eight separatrix loop (in red), two open bounded components foliated by closed curves and two elliptic fixed points (in blue).}
\end{center}
\end{figure}

\black
\subsubsection{Example: The Duffing oscillator}\label{sec:examples}

The Duffing oscillator corresponds to the potential 
\[
V(r)=\alpha r^2+\beta r^4.
\]
Below we consider the case $\alpha=1$ and $\beta=-1$.
\vspace{0.2cm}

\noindent\textbf{One-dimensional model $(N=1)$:}
The integrable Keplerian Hamiltonian \eqref{eq:keplern1} which governs the uncoupled dynamics reads
\[
h_{\mathrm{mech}}(r,y)=\frac{y^2}{2}-r^2+r^4.
\]
In the figure below we depict the phase space of this Hamiltonian. We identify three different open regions (two bounded and one unbounded) in which there exists a foliation by closed curves. The two bounded regions correspond to negative values of the energy while the unbounded one corresponds to positive energies. The three regions are separated by a figure eight homoclinic loop. The open unbounded region $M_2$ admits a foliation by two-dimensional invariant tori.  In any of these three regions, the potential $V(r)=r^2-r^4$ satisfies \textbf{A.1}, hence Proposition \ref{lem:uncoupledchangevarsLag} implies that $h_{\mathrm{mech}}$ satisfies \textbf{P1} for all  the invariant curves with Diophantine frequencies. Hence we can apply Theorem \ref{thm:mainellipticptsparticles} 
 to weak couplings of infinitely many of these Duffing Oscillators with exponentially decaying masses. 
\black 

\vspace{0.2cm}

\noindent\textbf{Two-dimensional model $(N=2)$:} The integrable Keplerian Hamiltonian \eqref{eq:uncoupledpolar}  which governs the uncoupled dynamics reads
\[
h_{\mathrm{mech}}(r,\mathcal R;\mathcal G)=\frac{\mathcal R^2}{2}+\frac{\mathcal G^2}{2r^2}-r^2+r^4.
\]

\noindent In the figure below we depict the phase space of this Hamiltonian. In polar variables the phase space is the direct product of the half plane $(r,\mathcal R)\in \mathbb R_+\times\mathbb R$ times the cylinder $(\alpha,\mathcal G)\in\mathbb T\times\mathbb R$.  Define the family of elliptic periodic orbits
\[
\gamma(\mathcal G)=\{r=r_*(\mathcal G),\mathcal R=0, \alpha \in \mathbb T\}
\]
where $r_*(\mathcal G)$ is the unique minimum of the effective potential contained in $\{r\geq 0\}$. Then, we can decompose the phase space $M=\{(r,\mathcal R,\alpha, \mathcal G)\in \mathbb R_+\times\mathbb R\times \mathbb T\times\mathbb R\}$ as 
\[
M=M_1\cup M_2\qquad\qquad\text{where}\qquad\qquad M_1=\bigcup_{\mathcal G\in\mathbb R}\gamma(\mathcal G)\quad\quad M_2=M\setminus M_1.
\]
The open unbounded region $M_2$ admits a foliation by two-dimensional invariant tori. The potential $V(r)=r^2-r^4$ satisfies \textbf{A.2}, hence Proposition \ref{lem:uncoupledchangevarsLag} implies that $h_{\mathrm{mech}}$ satisfies \textbf{P1} for all  the two-dimensional invariant tori with Diophantine frequencies. Hence we can apply Theorem \ref{thm:mainellipticptsparticles}
 to weak couplings of infinitely many of these Duffing Oscillators with exponentially decaying masses. 
\black 

\begin{figure}
\begin{center}
   \includegraphics[scale=0.4]{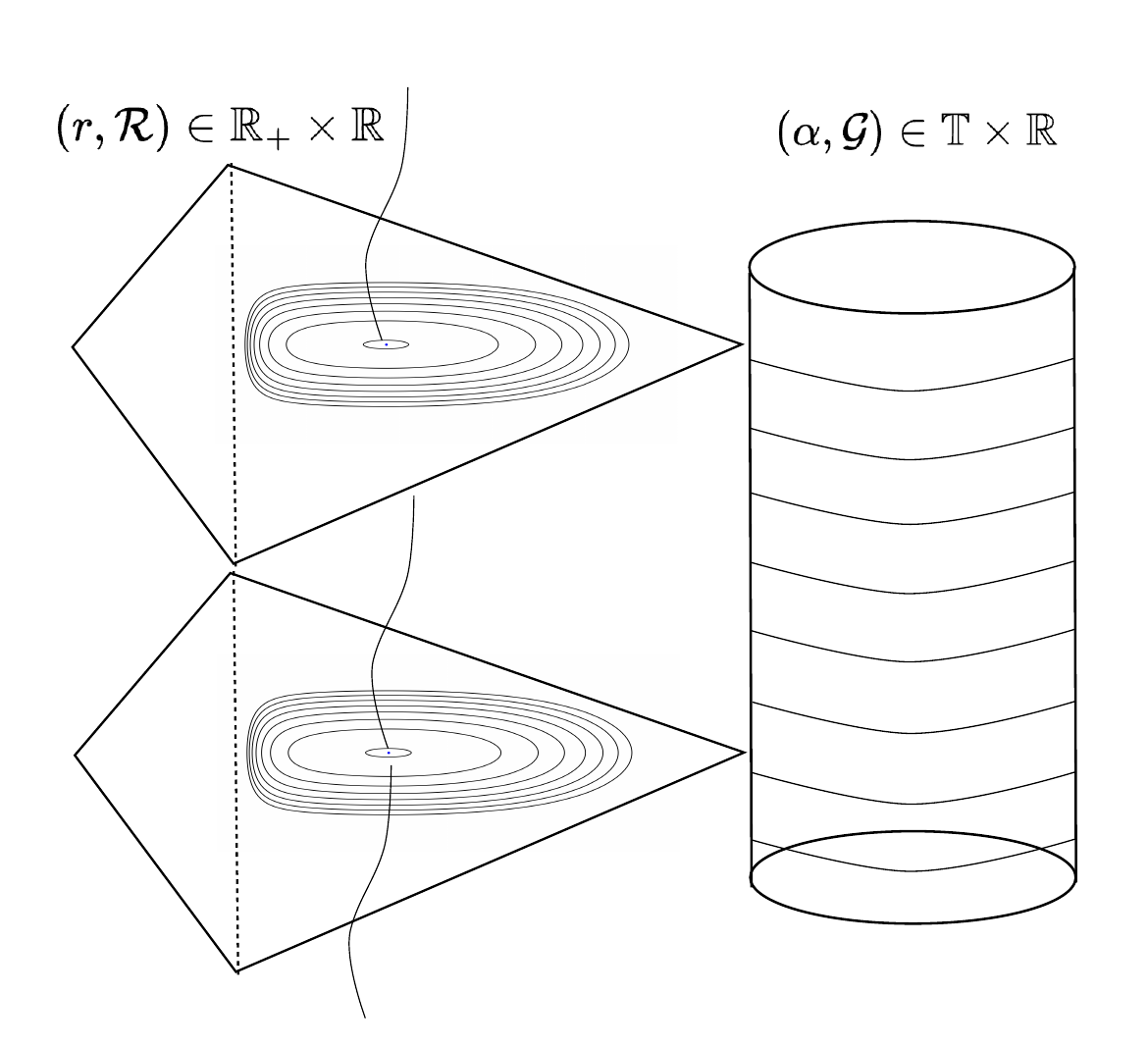}
\caption{The Duffing Hamiltonian on the plane ($N=2$). In the picture we have attached two different copies of the half plane $\mathbb R_+\times\mathbb R$  for different values of $\mathcal G$. On any of these copies the projection of the orbits can be decomposed into an unbounded open set foliated by closed curves and an elliptic fixed point. In the full phase space, the product of any of these closed orbits times a height level set of the cylinder yields a two-dimensional torus. The product of the elliptic fixed point times a height level set of the cylinder yields an elliptic periodic orbit.}
\end{center}
\end{figure}


\appendix

\section{Holomorphic mappings between Banach spaces}\label{sec:proofBanach}
Let $E,F$ be complex Banach spaces and let $U$ be an open subset of $E$. We say that a map $\Phi:U\to F$ is holomorphic if $\Phi$ is $\mathbb C$-differentiable. That is, for every $a\in U$ there exists a continuous linear map $A\in L(E,F)$ such that, for any small $v\in E$ 
\[
\Phi(a+v)-\Phi(a)-Av=o(\lVert v\rVert).
\]
We refer the interested reader to \cite{MujicaComplexBan} for a comprehensive introduction to the theory of holomorphic mappings between Banach spaces. For our purposes it will also be convenient to introduce the normed vector space
\begin{equation}\label{eq:Banachspacehol}
\mathcal H(U,F)=\{\Phi:U\to F\colon  \Phi \text{ is holomorphic and }\sup_{a\in U}\lVert \Phi(a)\rVert_F<\infty\}.
\end{equation}

\begin{lem}\label{lem:banachspaceholom}
The normed vector space $\mathcal H (U,F)$ is complete. Namely, it is a Banach space.
\end{lem}

\subsection*{Banach manifolds}
A Banach manifold is a set $M$ together with an (equivalence class) of atlas modeled on a Banach space (see, for instance,  \cite{DiffmfoldsLang}). A Banach manifold modeled in a Banach space $E$ is holomorphic if, for every two local charts $(U_i,\varphi_i)$ and $(U_j,\varphi_j)$ in its atlas, the map
\[
\varphi_i\circ\varphi_j^{-1}: \varphi_j(U_i\cap U_j)\subset E\to \varphi_i(U_i\cap U_j)\subset E
\]
is a holomorphic mapping. The following is a simple exercise which we reproduce for the sake of self-completeness.
\begin{lem}\label{lem:banachmanifold}
Let $\mathcal T_{\infty}=\ell_\infty/\sim$ where $\theta\sim\theta'$ if $\theta-\theta'\in\mathbb Z^{\mathbb N}$. Then
\begin{equation}\label{eq:bigphasespace}
\mathcal Q= \mathcal T_{\infty}\times\ell_\infty
\end{equation}
is a holomorphic Banach manifold.
\end{lem}
\begin{proof}
    As products of Banach manifolds are Banach manifolds, it is enough to check that $\mathcal T_{\infty}$ is a Banach manifold for which the transition maps between different local charts are real-analytic. Let $[\theta]\in\mathcal T_{\infty}=\{\tilde\theta\in\ell_\infty\colon \tilde\theta=\theta+k,\ k\in\mathbb Z\}$. Let $\rho<1/2$ and define  the local chart $(U_\rho([\theta]),\varphi_{[\theta]})$ with
    \[
    U_\rho([\theta]):=\{[\theta']\in\mathcal T_{\infty}\colon \inf_{\tilde\theta'\in[\theta']}|\theta-\tilde\theta'|_\infty<\rho\}
    \]
    together with the bijective map $ \varphi_{[\theta]} :U_\rho([\theta])\to \{|\tilde\theta|_\infty<\rho\}$ defined by
    \[
   \varphi^{-1}_{[\theta]}(\tilde\theta)=[\theta+\tilde\theta].
    \]
    But then, given any $[\lambda]\in\mathcal T_{\infty}$, for which $U_\rho([\theta])\cap U_{\rho'}([\lambda])\neq\emptyset$ we have that the transition map
$\varphi_{[\lambda]}\circ\varphi_{[\theta]}^{-1}:U_\rho([\theta])\cap U_{\rho'}([\lambda])\to \varphi_{[\lambda]}\left(U_\rho([\theta])\cap U_{\rho'}([\lambda])\right)$ is a translation (by a bounded vector depending on $\lambda,\theta$) so the proof follows.
\end{proof}

Next, we recall a technical lemma which was used in Section \ref{sec:Kamlagrangian}.
\begin{lem}[Lemma D.1. in \cite{PoschelSpatialstructure}]\label{lem:inverse}
    Let $O_h\subset\ell_\infty$ denote an open neighbourhood of radius $h$ of some real subset. Suppose that $f:O_h\subset\ell_\infty\to \ell_\infty$ is real-analytic. If 
     \[
     |f-\mathrm{id}|_\infty\leq\delta\leq  h/4
     \]
     then, $f$ has a real-analytic inverse $\varphi$ on $O_{h/4}$. Moreover, $|\varphi-\mathrm{id}|_\infty\leq \delta$ on this domain.
\end{lem}

\bibliographystyle{alpha}
\bibliography{biblioMelnikov}

\end{document}